\documentclass[10pt]{amsart}
\usepackage{amssymb,amsmath,amsfonts,amscd,mathrsfs, bbm, xcolor}

\newcommand{\R}{\mathbb{R}} 

\newcommand{\II}{\mathbbm{1}}

\newcommand{\pt}{\partial_t}
\newcommand{\FL}[1]{(-\Delta)^{#1}}
\newcommand{\invHL}{(-\Delta)^{-\frac12}}

\newcommand{\rbracket}[1]{\left(#1\right)}
\newcommand{\sbracket}[1]{\left[#1\right]}
\newcommand{\cbracket}[1]{\left\{#1\right\}}

\newcommand{\normp}[2]{\left\| #1 \right\|_{#2}}
\newcommand{\seminormp}[2]{\left[ #1 \right]_{#2}}
\newcommand{\abs}[1]{ \left| #1 \right|}

\newcommand{\Morrey}[2]{\widetilde{L}^{#1}(#2)}
\newcommand{\wMorrey}[2]{\widetilde{L}^{#1}_*(#2)}
\newcommand{\Campanato }[2]{\widetilde{\mathscr{L}}^{#1}(#2)}
\newcommand{\wCampanato }[2]{\widetilde{\mathscr{L}}^{#1}_*(#2)}
\newcommand{\pRieszP}[2]{\widetilde{I}_{#1}(#2)}

\newcommand{\bbN}{\mathbb{N}}

\newcommand{\bbR}{\mathbb{R}}
\newcommand{\bbS}{\mathbb{S}}

\newcommand{\calR}{\mathcal{R}}

\newcommand{\scrL}{\mathscr{L}}

\DeclareMathOperator{\diam}{\mathrm{diam}}

\renewcommand{\div}{\text{div}}

\DeclareMathOperator{\curl}{curl}

\DeclareMathOperator{\Id}{\mathbb{I}}

\def\Xint#1{\mathchoice
{\XXint\displaystyle\textstyle{#1}}%
{\XXint\textstyle\scriptstyle{#1}}%
{\XXint\scriptstyle\scriptscriptstyle{#1}}%
{\XXint\scriptscriptstyle
\scriptscriptstyle{#1}}%
\!\int}
\def\XXint#1#2#3{{%
\setbox0=\hbox{$#1{#2#3}{\int}$}
\vcenter{\hbox{$#2#3$}}\kern-.5\wd0}}

\def\dashint{\Xint-}

\renewcommand{\leq}{\leqslant}
\renewcommand{\geq}{\geqslant}
\renewcommand{\subset}{\subseteq}

\theoremstyle{plain}
\newtheorem{theorem}{Theorem}[section]

\newtheorem{lemma}[theorem]{Lemma}
\theoremstyle{definition}
\newtheorem{definition}[theorem]{Definition}

\theoremstyle{remark}
\newtheorem{remark}[theorem]{Remark}

\newcommand{\RR}{\mathbb{R}}

\begin{document}

%

%
\title[Surface flow for liquid crystals]{Liquid crystals and topological vorticity: smoothness of mild solutions}

\author[F.Lin]{Fanghua Lin}
\address{Courant Institute of Mathematical Sciences, New York University, NY 10012, USA}
\email{linf@cims.nyu.edu}
\author[Y. Sire]{Yannick Sire}
\address{\noindent
Department of Mathematics,
Johns Hopkins University, 3400 N. Charles Street, Baltimore, MD 21218, USA}
\email{ysire1@jhu.edu,ywu212@jhu.edu}

\author[Y. Wu]{Yantao Wu}

\author[Y. Zhou]{Yifu Zhou}
\address{\noindent
School of Mathematics and Statistics, Wuhan University, Wuhan 430072, China}
\email{yifuzhou@whu.edu.cn}

%

\begin{abstract}
We introduce several new models whose common feature is to take into account effects from topological vorticity. The macroscopic unknown is driven by a dissipative anomalous diffusion (of SQG-type) and is coupled with the orientation of the crystal, moving by the gradient flow of the energy of maps. The main idea of such models is to have a better insight on the vorticity formulation of the Liquid Crystal Flow and to tackle some regularity issues in the associated conserved  geometric motions. One of the advantage of the present PDEs is to capture features of the Navier-Stokes equations (or Euler) through a {\sl scalar} unknown, keeping the advection-diffusion structure of the orientation field. We obtain regularity for mild solutions under natural assumptions for the initial data, which are actually near-optimal. Along the way, we also draw some links with natural models of (anti-)ferromagnets previously investigated. 

\end{abstract}

\maketitle

\tableofcontents 

\section{Introduction}

{In this paper, we introduce several new models describing the flow of liquid crystals. Our starting point is the following system for the flow of a liquid crystal for nematic fluids: for $\Omega \subset \mathbb R^2$ a smooth domain, consider the system of PDEs
\begin{equation}\label{oLCF} \ \ \
\begin{cases}
\partial_t u + u\cdot \nabla u +\nabla P = \nu \Delta u - \lambda \nabla\cdot \left(\nabla d \odot \nabla d-\frac12 |\nabla d|^2 \mathbb{I}\right)~&\mbox{ in }~(0,T) \times \mathbb{R}^2,\\
\nabla\cdot u =0~&\mbox{ in }~(0,T) \times \mathbb{R}^2,\\
\partial_t d+u\cdot\nabla d=\gamma(\Delta d +|\nabla d|^2 d) ~&\mbox{ in }~(0,T) \times \mathbb{R}^2,\\
\end{cases}
\end{equation}
where $u:\Omega\times[0,T)\to\R^2$ is the fluid velocity field, $P:\Omega\times[0,T)\to\R$ is the fluid pressure function, $d:\Omega\times[0,T)\to\mathbb{S}^2$ stands for the orientation unit vector field of nematic liquid crystals and $\mathbb I$ is the identity matrix on $\R^2$. Note that the coupling involves the term $(\nabla d \odot\nabla d)_{ij}:=\nabla_i d\cdot\nabla_j d$. In the plane a natural quantity to consider instead of the velocity $u$ is the vorticity $\omega:=\nabla^\perp \cdot u$; however, it is clear that because of the coupling term in the first equation, the vorticity satisfies a PDE whose forcing is more singular than the one for the original unknown $u$. This in particular creates tremendous difficulties. On the other hand, several well-known equations reflecting geometric conserved motions involve (variations of ) the vorticity and it would be desirable to have a better understanding of such. One of our goal is to address such issues. More concretely, we derive and take up the analysis a wealth of various models taking the following form: 
\begin{equation}\label{HMF-SQG}
\left\{
\begin{aligned}
&\partial_t \theta +u\cdot \nabla \theta+ \nu (-\Delta)^a \theta= \lambda F(d,D^\beta d)~&\mbox{ in }~  (0,T) \times \mathbb{R}^2 \\
&u=\nabla^{\perp}(-\Delta)^{-1+\alpha}\theta,\quad \nabla\cdot u=0~&\mbox{ in }~(0,T) \times \mathbb{R}^2\\
&\partial_t d+u\cdot \nabla d= \gamma( \Delta d +|\nabla d|^2d) ~&\mbox{ in }~(0,T) \times \mathbb{R}^2\\
\end{aligned}
\right.
\end{equation}
where the parameters satisfy $0\leq a,~\alpha\leq 1$. The coupling term $F(d,D^\beta d)$ is function of the orientation field $d$ and some of its derivatives (as we will see below, $\beta$ only depends on $\alpha$).  The temperature $\theta$ defined on $\mathbb R^2\times (0,T)$ is a scalar, the velocity field $u: \mathbb R^2 \to \mathbb R^2$ depends constitutively on $\theta$, and the field $d: \mathbb R^2\times (0,T) \to \mathbb S^2$ describes the orientation of the crystal.
 Three positive constants $\nu $, $\lambda$, and $\gamma$ respectively quantifies viscosity, the competition between kinetic energy and elastic energy, microscopic elastic relaxation time for the director field. Here we assume that $\lambda=\gamma$ such that we have conservation property in {\it basic energy law} \eqref{basicenergy}. Without loss of generality, we assume that $\nu = \lambda = \gamma=1$  since the exact values of these constants play no role in our qualitative results.

The formulation of the original liquid crystal system in terms of the vorticity allows effectively to introduce several parameters in the PDEs, which tune the strength of diffusion and the Biot-Savart law. The surface quasi-geostrophic equation (SQG) then appears naturally. The systems then consist in a heat equation of the orientation (the heat flow of harmonic maps) coupled with  a {\sl scalar} equation. Those models exhibit several scales reminiscent of criticality. The derivation of system \eqref{HMF-SQG} leads to a {\sl $2-$parameter} and one can then exhibit different regimes depending on the order of the diffusion $(-\Delta)^a$ and the order of the forcing term $F(d,D^\beta d)$. We always choose the Biot-Savart law to be of the form $u=\nabla^{\perp}(-\Delta)^{-1+\alpha}\theta$ (which amounts to a generalized SQG relation). The requirement to satisfy dissipation of energy implies some admissible structures for the forcing $F$ whose order have to be measured with  respect to diffusion. 
  
\smallskip
 
It should be noted that the values of the parameters $a,\alpha$ play an important role in the analysis of system \eqref{HMF-SQG}. Indeed, it is straightforward to see that the system \eqref{HMF-SQG} involves the Euler equation in vorticity whenever $a=\alpha=0$; the modified inviscid SQG for $a=0<\alpha<1/2$; and the modified dissipative SQG for $0<\alpha<1/2$. The dissipative/inviscid surface quasi-geostrophic equation has been introduced in \cite{CMT} as a toy model for the regularity of Navier-Stokes equations. For $d \geq 2$, dissipative SQG is the following drift-diffusion equation given by 
\begin{align*}\label{eq1}
&\partial_t \theta + u\cdot \nabla \theta+ (-\Delta)^a \theta = F ~~\text{ in }~~(0,T)\times\Omega \subset \mathbb{R} \times\mathbb{R}^d,
\end{align*}
where as before $\theta:(0,T)\times\mathbb{R}^2 \to \mathbb{R}$, $u= \nabla^{\perp}(-\Delta)^{-1+a}\theta :  (0,T)\times\mathbb{R}^2 \to \mathbb{R}^2$ and $F: (0,T)\times\mathbb{R}^2 \to \mathbb R$ is a possibly singular forcing.  Dissipative SQG is subcritical for $a>1/2$, critical for $a=1/2$ and supercritical otherwise. Note that the motion is incompressible since $\div \, u=0$. We refer the reader to the nice survey by Kiselev  \cite{kiselev} for a discussion of the equation, recent results and open questions. In particular, the critical case (for regularity) has attracted a lot of attention in the recent years after the breakthrough of Caffarelli and Vasseur (see \cite{CaffaVasseur,KNV,KN}). 

\smallskip
 
 In Equation \eqref{HMF-SQG}, it is clear that the right hand side of the first equation acts as a forcing term on SGQ. In view of this very particular structure, we introduce the following trichotomy: 
\begin{itemize}
\item The {\sl critical} case corresponds to the values $a=\alpha=1/2$;
\item The {\sl subcritical } range corresponds to $a>\alpha=1/2$;
\item The {\sl supercritical } range corresponds to $a < \alpha=1/2$. 
\end{itemize}

 For previous results on nonlocal elliptic and parabolic equations without drift see for instance \cite{KassCalcVar,CCV,FelsKass,DKP,Coz17,KaWe22a}. Let us mention also the articles \cite{SilvestreAdv, SilvestrePisa, SilvestreIndiana}, where H\"older estimates have been obtained for nonlocal drift-diffusion equations using non-variational techniques.

The theory of liquid crystals has been developed in many seminal works by  Ericksen \cite{Ericksen1962ARMA} and Leslie
\cite{Leslie1968ARMA}.  In \cite{L1989CPAM}, the first author considered a simplified model (see references therein for previous works) that is precisely the system \eqref{oLCF}.  Though simplified, it still enjoys the same type of energy law, coupling structure and dissipative properties as those latter works. For the study of the nematic liquid crystal flows with Dirichlet or Neumann boundaries, there have been growing interests concerning the global existence of weak solutions, partial regularity results, singularity formation and others (see e.g. \cite{L1989CPAM,LL1995CPAM,LL1996DCDS,LLW10,LW2014RSTA,LW2016CPAM,HLLW2016ARMA,LCF2D} and references therein). For instance, for the  Dirichlet boundary condition with $u\equiv0$ on $\partial\Omega$ and $d|_{\partial\Omega}=d_0\in C^{2,\beta}(\partial\Omega,\bbS^2)$,  the following properties for weak solutions $(d,u)$ have been proved: global (in time) existence of weak solutions, smoothness except for finitely many singular time and a discrete set of points, interior energy concentration at singular time,  characterization of tangent map, asymptotic decomposition of energy density measure (see e.g. \cite{LLW10} in two dimensions). 

\medskip

The analysis of system \eqref{oLCF} on planar domains is by now relatively well developed; however, the existence of weak solutions in higher dimensions, together with their (partial) regularity is still vastly open, despite some recent important results. On the other hand, the so-called {\sl topological vorticity} and the associated conserved geometric motions exhibit great analytical challenges and only partial results are available. The purpose of this paper is in particular to derive several PDEs involving vorticity, and as a first step towards their analysis prove global  smoothness under natural conditions on the unknowns.

We now describe a bit more in details the PDEs we consider. The forcing term $F(d,D^\beta d)$ can be chosen as 

\begin{equation*}
 F_1(d,D^\beta d):=(-\Delta)^{-1+\alpha} \nabla^\perp \nabla\cdot\left(\nabla d\otimes \nabla d -\frac12|\nabla d|^2\mathbb I_2\right),
\end{equation*}
and the quantity $\displaystyle{\mathcal E_1(t):=\int_{\mathbb R^2} | \theta|^2 +|\nabla d|^2}$ will be dissipated. Alternatively, the forcing term can be chosen as
\begin{equation*}
 F_2(d,D^\beta d):=\nabla^\perp \nabla\cdot\left(\nabla d\otimes \nabla d -\frac12|\nabla d|^2\mathbb I_2\right)
\end{equation*}
and the dissipated energy is  $\displaystyle{\mathcal E_2(t):=\int_{\mathbb R^2} |(-\Delta)^{(\alpha-1)/2} \theta|^2 +|\nabla d|^2.}$ In the hydrodynamic theory of liquid crystals, the stress tensor defined by
$$
\Xi :=\nabla d\odot\nabla d -\frac12|\nabla d|^2\Id_2,
$$
plays an important role since it is related to {\it total topological vorticity} of the system. 

\smallskip

It is important to notice that the L. H. S. of \eqref{HMF-SQG} is of order $2a$, while the R.H.S. is of order $1+2\alpha$ whenever $F_1(d,D^\beta d)=(-\Delta)^{-1+\alpha} \nabla^\perp \nabla \cdot \,\Xi$ and of order $3$ whenever $F_2(d,D^\beta d)=\nabla^\perp \nabla \cdot \,\Xi$. 

\smallskip

In order to illustrate features of our models and initiate their analysis, we choose here the following {\sl ad hoc} values for the parameters: $a \in (\frac12,1)$ and $\alpha=\frac12$ and the forcing to be 
$F(d,D^\beta d)=(-\Delta)^{-\frac12 } \nabla^\perp \nabla \cdot \,\Xi$. This amounts to investigate 
\begin{equation}\label{eqn: SQGHMF}
\begin{cases}
\pt \theta + u\cdot\nabla\theta +\FL{a} \theta =  (-\Delta)^{-\frac12} \nabla^\perp \nabla \cdot \,\Xi, \\
u = \calR^\perp \theta = \nabla^\perp (-\Delta )^{-\frac12} \theta, \quad \nabla\cdot u =0, \\
\pt d + u\cdot\nabla d = \Delta d + |\nabla d|^2d.
\end{cases}
\end{equation}
In the previous system, the fractional laplacian $\FL{a} $ stands for the Fourier multiplier of symbol $|\xi|^{2a}$, also being defined by the hypersingular integral (for suitable functions $\theta$)
\[ \FL{a}\theta(x) = c(n,a) P.V. \int_{\bbR^2} \frac{\theta(x)-\theta(y)}{|x-y|^{2+2a}} dy. \]

\smallskip

In Section 2, we provide a rationale for the  derivation of  \eqref{HMF-SQG} from the original system \eqref{oLCF} formulated in vorticity. 
The operator $\mathcal R^\perp$ is the (rotated) vectorial Riesz transform, which maps $L^p(\bbR^2)$ into itself for any $p \in (1,\infty)$. Dealing with the existence of global weak or mild solutions for \eqref{eqn: SQGHMF} amounts to find the natural space of initial data $(\theta_0,d_0)$ such that the term  $(-\Delta)^{-\frac12}\nabla^\perp \nabla \cdot \,\Xi $ can be seen as a suitable forcing term for the dissipative SQG equation. Notice that the energy dissipation holds for weak solutions; and in the present contribution we consider specifically mild solutions, i.e. semi-group solutions. The reason why we consider this notion of solution is to make the presentation more transparent. In virtue of the actual state of art for the full system \eqref{oLCF} we expect that weak solutions in two-dimensional domains for \eqref{eqn: SQGHMF} should exist globally in time and be even partially smooth. We postpone such results to future installments of our investigation.

\medskip

\subsection*{Main results}
The aim of the present work is to prove global regularity for {\sl mild} solutions of system \eqref{eqn: SQGHMF}. Mild solutions are basically semi-group solutions and we refer to Section 3 for a precise definition. For the moment, we keep in mind that such solutions are defined in an integral sense using the heat kernel of $\partial_t +(-\Delta)^a$. 
{We would like to stress that we do not prove the existence of mild solutions in this paper. Notice that in the original Liquid Crystal Model \eqref{oLCF} the forcing term (the R.H.S. of the first equation) has the same number of derivatives as the main term of the equation. This is not necessarily the case for the models \eqref{HMF-SQG}. In the particular situation we consider $a \in (1/2,1)$ and $\alpha=1/2$, the coupling term $F$ has two derivatives where as the main term has $2a<2$. This creates of course several complications. This is one of the reasons why we focus here on conditional global regularity.  It should be noted that the existence theory of weak or mild solutions, together with their partial regularity, is not at all an easy problem since the forcing term in the fluid equation is rather singular.  }

We would like to emphasize the following two crucial points on our systems:

\begin{itemize}
\item The vorticity formulation allows to reduce the fluid equation to a scalar one. 
\item The forcing term has one more derivative than in the original system \eqref{oLCF}. 
\end{itemize}

Our main result is the following. 

\begin{theorem}\label{Thm: Reg-intro}
Assume that a mild solution of \eqref{eqn: SQGHMF} satisfies  $\theta_0\in L^{\frac{2}{2a-1}}(\bbR^2), ~\nabla d_0\in L^2(\bbR^2),~ \theta\in L^2\cap L^{p_0}((0,T)\times \bbR^2)$ for some $p_0>\frac{2a+2}{2a-1}$, and $|\nabla d| \in L^2\cap L^{q_0}((0,T)\times \bbR^2)$ for some $q_0>4$,
then this solution to \eqref{eqn: SQGHMF} is classical: $\theta,u,d\in C^\infty((0,T)\times\bbR^2)$.
\end{theorem}

\smallskip

 The regularity result in Theorem \ref{Thm: Reg-intro} is near optimal in the following sense. Recall that for original version of nematic liquid crystal flow \eqref{oLCF}, it is shown in \cite{LLW10} that the assumptions of the  {\sl weak} solution  to be as $u,~\nabla d\in L^4((0,T)\times\bbR^2)$ indeed gives  smoothness. On the other hand, one can formally view original liquid crystal flow \eqref{oLCF} (in vorticity formulation) as the special case of \eqref{eqn: SQGHMF} with $a=1$ (and in our case we even have that $u$ is the Riesz transform of $\theta$ and then $u$ and $\theta$ have the same integrability by boundedness of singular integrals). Our theorem then generalizes the regularity result to {\sl subcritical} fractional Laplacian in a near optimal way: we need $p_0>\frac{2a+2}{2a-1},~q_0>4$.

\smallskip

To prove Theorem \ref{Thm: Reg-intro}, we argue in three steps we now describe since they are of independent interest. The first step consists in arguing higher integrability through the following Theorem.

\begin{theorem}\label{Prop: Reg1-intro}
Assume that mild solution of \eqref{eqn: SQGHMF} satisfies  $\theta_0\in L^{\frac{2}{2a-1}}(\bbR^2), ~\nabla d_0\in L^2(\bbR^2), ~\theta\in L^2\cap L^{p_0}((0,T)\times \bbR^2)$ for some $p_0>\frac{2a+2}{2a-1}$, and $|\nabla d| \in L^2\cap L^{q_0}((0,T)\times \bbR^2)$ for some $q_0>4$,
then we can conclude that $\theta~,u,~\nabla d \in L^p((0,T)\times \bbR^2)$ for any $p\in[2,\infty)$.
\end{theorem}

We now prove that this leads to zero order regularity, i.e. H\"older regularity. 

\begin{theorem}\label{Prop: Reg2-intro} 
Assume that mild solution of \eqref{eqn: SQGHMF} satisfies  $\theta_0\in L^{\frac{2}{2a-1}}(\bbR^2), ~\nabla d_0\in L^2(\bbR^2),~ \theta\in L^2\cap L^{p_0}((0,T)\times \bbR^2)$ for some $p_0>\frac{2a+2}{2a-1}$, and $|\nabla d| \in L^2\cap L^{q_0}((0,T)\times \bbR^2)$ for some $q_0>4$,
 then we can conclude that $\theta~,u\in C^\alpha_{loc}((0,T)\times\bbR^2)$ for some $\alpha>0$.
\end{theorem}

A standard bootstrap argument then leads to the smoothness of mild solutions and our main theorem.  

\smallskip

It is worth mentioning that in the case of arbitrary Dirichlet type boundary condition, one can only hope to obtain $C^\alpha_{t,x}$ instead of $C^\infty_{t,x}$ smoothness, because fractional Laplacian allows jumps of $\partial_t\theta, \theta$ even in the simplest problem of fractional heat equation with Dirichlet data. We work on improved regularity such as $\nabla\theta\in C^\alpha_{x,t}$ only for the whole space $\bbR^2$ but would like to mention that via a by-now quite standard argument, one can get interior estimates for the boundary value problem. 

\medskip

The systems we investigate in this paper share many similarities with other systems present in the literature. In particular, our techniques allow to prove as well some smoothness for such systems. We would like to emphasize first that in the two-dimensional case the heat flow of harmonic maps is actually the main contributor to the formation of singularities. Of course higher dimensional versions will see the macro-equation (of fluids) strongly compete with the harmonic map part. The singularity of the forcing term will also play a significant role for the development of singularities in any dimension.  
As we already mentioned, the PDEs we introduce here are reminiscent of several variations of fluid equations, which amount to consider the multiplier $(-\Delta)^a$ as the diffusion term. Similarly, as for the structure of the coupling term for instance, our model is also related to the so-called Harmonic-Ricci flow (see e.g. \cite{MR3653093,MR2961788}) which can be seen as a variation of the Ricci-De Turck flow. Assume $(M,g)$ is a smooth closed oriented surface and let $g=g(t)$ be a smooth evolving metric on $M$ and consider a map $\phi:~M\to \mathbb S^2$.  Buzano's thesis \cite{BuzanoThesis} considered 
\begin{equation}\label{HRF}
\left\{
\begin{aligned}
&g_t=-2{\rm Ric}_g+2\alpha \nabla\phi\otimes\nabla\phi,\\
&\phi_t= \Delta_{g(t)} \phi + |\nabla_{g(t)} \phi |^2 \phi.\\
\end{aligned}
\right.
\end{equation}
In particular, considering as unknown the Gauss curvature $K$ instead of the metric itself, and owing to classical differential formulas on surfaces, 
the previous system writes (in its volume-preserving formulation) 
\begin{equation}\label{eqn-Gauss}
\left\{
\begin{aligned}
&~K_t=\Delta_g K+K^2+\alpha \nabla^\perp \nabla\cdot\left(\nabla \phi\otimes \nabla \phi-\frac12 |\nabla \phi|^2 \mathbb I_2\right),\\
&~\phi_t=\Delta_g \phi+|\nabla_g \phi|^2\phi . 
\end{aligned}
\right.
\end{equation}
This resembles the vorticity formulation in the nematic liquid crystal flow \eqref{oLCF} (see next section for more details). Then one can proceed using the conformal change to recover the original quantities
$$
g=e^{2u} g_0,\quad K_g=e^{-2u}(K_{g_0}-\Delta_{g_0} u).
$$

In \cite{MR3653093}, Buzano and Rupflin introduced a volume-preserving version of \eqref{HRF} (much related to \eqref{eqn-Gauss} ), hinting towards a dichotomy blow-up/global existence depending on the strength of the coupling constant $\alpha$. See also related Teichm\"uller theory of harmonic maps for a similar system \cite{wolf1989teichmuller,tromba2012teichmuller,rupflin2016flowing,rupflin2018teichmuller}.
\smallskip

Finally, we would like to mention a last system sharing also similar features. In the theory of (anti-)ferromagnets with easy plane anisotropy the equations of the spin density $m=(m_1,m_2,m_3):\mathbb R^2 \to \mathbb S^2$ can be formulated as an infinite dimensional Hamiltonian system with Hamilton function 
$$
E(m)=\int_{\mathbb R^2} |\nabla m|^2+ m^2_3
$$ 
and formal symplectic structure 
$$
\Omega=m \cdot \partial_1 m \wedge \partial_2 m. 
$$
The form $\Omega$ is the topological vorticity and introducing the stress tensor 
$$
\tilde \Xi_{ij}:= \frac{|\nabla m|^2+m_3^2}{2}\delta_{ij}-\partial_i m \cdot \partial_j m
$$
for $1 \leq i,j \leq 2$, the conservation law for the topological vorticity becomes
$$
\partial_t \Omega = - \nabla^\perp \nabla \cdot \tilde \Xi. 
$$
These systems have been investigated by Komineas and Papanicolaou for instance in \cite{MR1420808} (see also \cite{Lin-Shatah, Hang-Lin} and references therein). A feature of this system is that one can penalize the constraint of being in the sphere (in view of the term $m_3$ in the energy ) by forcing the spin density to take values in the hemisphere $\mathbb S^2_+$. It is then natural in view of the previous discussion, and the derivations in Section 2, to consider actually the following system
\begin{equation}\label{ferromag_eps} \ \ \
\begin{cases}
\partial_t \Omega + u\cdot \nabla \Omega +(-\Delta)^{a} \Omega = (-\Delta)^{-\frac12}\nabla^\perp \nabla \cdot \,\tilde \Xi ~&\mbox{ in }~(0,T) \times \mathbb{R}^2,\\
\nabla\cdot u =0~&\mbox{ in }~(0,T) \times \mathbb{R}^2,\\
\partial_t d+u\cdot\nabla d=\Delta d +(|\nabla d|^2 +\frac{d_3^2}{\varepsilon^2})d-\frac{d_3}{\varepsilon^2}e_3  ~&\mbox{ in }~(0,T) \times \mathbb{R}^2,\\
\end{cases}
\end{equation}

\medskip

In the papers \cite{Lin-Shatah,Hang-Lin} the analysis of the ferromagnet system alone (without the coupling as above) has been investigated. However, as far as we know, the previous system has not been introduced in the literature. The techniques we investigate in the present work can lead to regularity as well. 

The paper is organized as follows: in Section 2, we justify the PDE system in consideration by deriving first the original liquid crystal flow in vorticity formulation in two dimensions and then phenomenologically introducing the dissipative SQG. The nonlinear term is then built such that the {\it basic energy law} \eqref{basicenergy} holds. Section 3 is devoted to the proofs of our main results.  We finally collect in the appendix several crucial lemmata on potential-theoretic aspects of Morrey spaces in the parabolic setting. 

\medskip

\section{Derivation of the system}

This section is devoted to some details about the derivation of our new models, and some motivation behind it. The starting point is the liquid crystal flow system on $\bbR^2$ introduced by Lin \cite{L1989CPAM}, namely \eqref{oLCF}. More precisely, we write it in vorticity formulation by taking the $\curl$ of the first equation. Introducing the vorticity as the {\sl scalar} $\omega=\curl u=-\partial_{x_2} u_1+\partial_{x_1} u_2$, a standard computation gives

\begin{equation}\label{vor_NS}
\omega_t+u\cdot\nabla\omega-\Delta\omega=-\curl\div(\nabla d\odot\nabla d -\frac12|\nabla d|^2\Id_2)
\end{equation}

Denoting $ \Xi$ the stress-energy tensor 
$$
\Xi :=\nabla d\odot\nabla d -\frac12|\nabla d|^2\Id_2,
$$
Equation \eqref{vor_NS} is reminiscent of the following formulation of liquid crystals
$$
\frac{\partial \Omega}{\partial t}= -\curl\div \, \Xi,
$$
where $\Omega$ is the so-called {\it topological vorticity}. Postulating now that the velocity field $u$ whose curl is the vorticity $\omega$ depends in a nonlocal fashion w.r.t. the temperature $\theta$ as
\begin{equation*}
u:=\nabla^{\perp}(-\Delta)^{-1+\alpha}\theta, 
\end{equation*}
with $\alpha \in (0,1)$, the left hand side of \eqref{vor_NS} is chosen to be  
$$
\partial_t\theta + u\cdot\nabla\theta + (-\Delta)^a \theta 
$$
where $a \in (0,1)$ is another real parameter. We then obtain a general system of the form \eqref{HMF-SQG}. 
\begin{remark}
In the next section, one also could derive a model coupling hypo/hyper-dissipative versions of Navier-Stokes with the heat flow of maps. This amounts to consider  
$$
\partial_t v + v\cdot \nabla v +\nabla P = -\nu (-\Delta)^a v 
$$
for either $a<1$ or $a>1$. This model and its singularity formation has been investigated  in the literature (see e.g. \cite{MR4057903,MR4057903} and references therein).  
\end{remark}

\subsection{The forcing term}

We now derive the structure of the forcing term $F(d,D^\beta d)$. An important feature of the PDE \eqref{oLCF} is the dissipation of the formal energy
$$
E(t):=\int_{\mathbb R^2} |u|^2 +|\nabla d|^2. 
$$ 
Namely, the coupling is chosen such that the following dissipation holds for \eqref{oLCF}
\begin{equation*}
\frac12 \frac{d}{dt}\left(\int |\nabla d|^2+|u|^2\right)= -\int \big|\Delta d +|\nabla d|^2d\big|^2-\int \big| \nabla u\big|^2 \leq 0 .
\end{equation*}
This is the crucial {\it basic energy law}.  It relies on designing the forcing term such that cancellations occur whenever the PDE is tested with the stress tensor of the orientation $\Delta d+|\nabla d|^2 d$.  Testing \eqref{HMF-SQG}$_3$ with $\Delta d +|\nabla d|^2d$ we obtain
\begin{equation}\label{eqnheat}
\begin{aligned}
-\frac12 \frac{d}{dt}\int |\nabla d|^2+\int (u\cdot \nabla d)\Delta d=\int \big|\Delta d +|\nabla d|^2d\big|^2,
\end{aligned}
\end{equation}
where we have used
\begin{equation*}
\begin{aligned}
\int (u\cdot \nabla d)|\nabla d|^2d=\int u_i\partial_i d_j d_j |\nabla d|^2=\int u\cdot \nabla\left(\frac{|d|^2}{2}\right)|\nabla d|^2=0.
\end{aligned}
\end{equation*}

Now, we consider two types of energy dissipation:  
$$
\mathcal E_1(t):=\int_{\mathbb R^2} | \theta|^2 +|\nabla d|^2. 
$$ 

and

$$
\mathcal E_2(t):=\int_{\mathbb R^2} |(-\Delta)^{(\alpha-1)/2} \theta|^2 +|\nabla d|^2. 
$$ 

The previous two quantities $\mathcal E_1$ and $\mathcal E_2$ control two different norms of $\theta$ and lead to two different models. 

\subsection{Ad hoc singular forcing}
Testing  \eqref{HMF-SQG}$_1$  by $\theta$, using that $u$ is divergence free,  we get
\begin{equation}\label{eqn1}
\begin{aligned}
\frac12\frac{d}{dt}\int \big | \theta \big |^2=-\int \big | (-\Delta)^{\frac{ \alpha -1 +a}{2}}\theta \big|^2+\int F(d,D^\beta d)\theta . 
\end{aligned}
\end{equation}
 From \eqref{eqn1} and \eqref{eqnheat}, we have
\begin{equation*}
\begin{aligned}
\frac12 \frac{d}{dt}\left(\int |\nabla d|^2+|\theta|^2\right)=&~-\int \big|\Delta d +|\nabla d|^2d\big|^2-\int \big|(-\Delta)^{\frac{ \alpha -1 +a}{2}}\theta\big|^2\\
&~+ \int F(d,D^\beta d)\theta +\int (u\cdot \nabla d)\Delta d. \\
= &~-\int \big|\Delta d +|\nabla d|^2d\big|^2-\int \big|(-\Delta)^{\frac{ \alpha -1 +a}{2}}\theta\big|^2 \leq 0
\end{aligned}
\end{equation*}
provided the following balance law holds 
\begin{equation}\label{basicenergy}
\int_{\RR^2} F_1(d,D^\beta d)\theta  + \int_{\RR^2} (u\cdot \nabla d) \Delta d = 0 .
\end{equation}

Using now the constitutive relation $u=\nabla^{\perp}(-\Delta)^{-1+\alpha}\theta$, we compute
\begin{equation*}
\begin{aligned}
\int (u\cdot \nabla d)\Delta d=&~\int u_1\partial_1 d_j \partial_{kk}d_j+ \int u_2\partial_2 d_j \partial_{kk} d_j\\
=&~\int \partial_1 d_j \partial_{kk}d_j\partial_2\big[(-\Delta)^{-1+\alpha}\theta\big]-\int \partial_2 d_j \partial_{kk} d_j\partial_1\big[(-\Delta)^{-1+\alpha}\theta\big]\\
=&~-\int (-\Delta)^{-1+\alpha} \partial_2\big[\partial_1 d_j \partial_{kk}d_j\big] \theta+\int (-\Delta)^{-1+\alpha} \partial_1\big[\partial_2 d_j \partial_{kk} d_j\big]\theta\\
=&~-\int ( -\Delta)^{-1+\alpha} \curl \div (\nabla d\odot \nabla d-\frac12|\nabla d|^2\Id_2) \theta\\
= &~ - \int F_1(d,D^\beta d)\theta
\end{aligned}
\end{equation*}
This gives rise to system \eqref{HMF-SQG} with 
$$
F_1(d,D^\beta d):= ( -\Delta)^{-1+\alpha} \curl \div \Xi .
$$
This system dissipates $\mathcal E_1$. 

\medskip

\subsection{Topological vorticity formulation }

Testing now  \eqref{HMF-SQG}$_1$  with the multiplier $(-\Delta)^{\alpha-1}\theta$, using that $u$ is an orthogonal gradient (and hence $\int  (u \cdot \nabla \theta )(-\Delta)^{\alpha-1}\theta=0$)  we get
\begin{equation}\label{eqn2} \ \ \
\begin{aligned}
\frac12\frac{d}{dt}\int \big | (-\Delta)^{\frac{ \alpha -1}{2}} \theta \big |^2=-\int \big | (-\Delta)^{\frac{ \alpha -1 +a}{2}}\theta \big|^2+\int F_2(d,D^\beta d)(-\Delta )^{\alpha-1}\theta . 
\end{aligned}
\end{equation}
 From \eqref{eqn2} and \eqref{eqnheat}, we have
\begin{equation*}
\begin{aligned}
\frac12 \frac{d}{dt}\left(\int |\nabla d|^2+|(-\Delta)^{\frac{\alpha-1}{2}}\theta|^2\right)=&~-\int \big|\Delta d +|\nabla d|^2d\big|^2-\int \big|(-\Delta)^{\frac{ \alpha -1 +a}{2}}\theta\big|^2\\
&~+ \int F_2(d,D^\beta d)(-\Delta )^{\alpha-1}\theta +\int (u\cdot \nabla d)\Delta d. \\
= &~-\int \big|\Delta d +|\nabla d|^2d\big|^2-\int \big|(-\Delta)^{\frac{ \alpha -1 +a}{2}}\theta\big|^2 \leq 0
\end{aligned}
\end{equation*}
provided the following balance law holds 
\begin{equation*}
\int_{\RR^2} F_2(d,D^\beta d)(-\Delta )^{\alpha-1}\theta  + \int_{\RR^2} (u\cdot \nabla d) \Delta d = 0 .
\end{equation*}

Using now the constitutive relation $u=\nabla^{\perp}(-\Delta)^{-1+\alpha}\theta$, we compute as before to get 
$$
F_2(d,D^\beta d)=\curl \div \, \Xi .
$$
We then have derived \eqref{HMF-SQG} with then previous forcing, dissipating the energy $\mathcal E_2$. 

\medskip

\section{Regularity of mild solutions}

This section is devoted to the proofs of the regularity results. We first collect the following estimates on the heat kernel of $\partial_t+(-\Delta)^a$. Let $p_a(t,x-y)$ to be the fundamental solution to fractional heat equation $\partial_t w + \FL{a} w= 0$ on $(0,\infty)\times\bbR^2$. It is well known that $p_a\in C^\infty((0,\infty)\times \bbR^2)$ and $p_a(t,x)=p_a(t,-x)$ for any $(t,x)\in(0,\infty)\times\bbR^2$. Moreover, we have the following scaling and semigroup properties:
\begin{align*}
p_a(t,x) & = t^{-\frac1a}p_a\rbracket{1, t^{-\frac{1}{2a}} x}, \ t>0, ~x\in\bbR^2 \\
p_a(t,x) & = \int_{z\in \bbR^2} p_a(t-s,x-z) p(s,z), \  t>s>0, ~x\in\bbR^2
\end{align*}

We also have the following pointwise estimate, see \cite{Jakubowski20}
\begin{equation}\label{Est: HeatKernel} \begin{split}
|\nabla^k p_a|(t,x), |\calR_i \nabla^k p_a|(t,x) & \lesssim \frac{1}{\delta((0,0),(t,x))^{2+k}}\\
|\partial_t p_a|(t,x),|\partial_t \calR_i p_a|(t,x) &\lesssim \frac1t \frac{1}{\delta((0,0),(t,x))^{2}} 
\end{split}\end{equation}
where $\delta((s,y),(t,x))=\max\rbracket{|s-t|^{\frac1{2a}},|x-y|}$.

We first define the notion of solutions we consider.

\smallskip

\begin{definition}\label{def-mild}
Let $a \in (0,1)$. 
Denote $p_a$ the fundamental solution of the parabolic  operator $\partial_t + \FL{a}$ on $\bbR^2$ and $P_{a,t}$ the (stable) semigroup operator
$$P_{a,t}f(x)=\int_{\bbR^2}p_a(t,x-y)f(y)dy, \ \ \  t>0, ~x\in\bbR^2 . $$
We call a mild solution to $\eqref{eqn: SQGHMF}$ a couple  $(\theta,d) \in L^1_{loc} ((0,T) \times \mathbb R^2)^2$, satisfying a.e. 
\begin{equation*} 
\theta(t,x)   = P_{a,t}\theta_0(x) +
\end{equation*}
\begin{equation*}
 \int_{s\in(0,t)} \int_{y\in\bbR^2}  p_a(t-s,x-y)\sbracket{ (-\Delta)^{-\frac12}\nabla^\perp \nabla \cdot \,\Xi- u\cdot\nabla\theta }(s,y),
\end{equation*}
whenever the integrals are converging and such that $u = \mathcal R^\perp \theta .$

\end{definition}

\subsection{Higher integrability: Proof of Theorem \ref{Prop: Reg1-intro}}

Recall that  by Duhamel principle, a mild solution to $\eqref{eqn: SQGHMF}_1$ can be written as
\begin{align*} 
\theta(t,x)  & = P_{a,t}\theta_0(x) + \int_{s\in(0,t)} \int_{y\in\bbR^2}  p_a(t-s,x-y)\sbracket{ (-\Delta)^{-\frac12}\nabla^\perp \nabla \cdot \,\Xi - u\cdot\nabla\theta }(s,y) \\
& = P_{a,t}\theta_0(x)  + J_1(t,x) + J_2(t,x). 
\end{align*}
 
It will be useful to work out some algebraic structure of the forcing term, namely
\begin{align}\label{eqn: F(d)} \ \ \ \ \ \ \ \
\begin{split}
-\curl\div \, \Xi &  =- \invHL(\partial_2 d_j \partial_{1kk}d_j - \partial_1 d_j \partial_{2kk}d_j)\\
& = -\partial_{1k} \invHL (\partial_2 d_j \partial_k d_j) + \partial_{2k} \invHL (\partial_1 d_j \partial_k d_j) .\\
\end{split} \end{align}

Moreover, if we denote
\[ \tau(d):=\partial_t d + u\cdot\nabla d = \Delta d + |\nabla d|^2d \]
to be the deviation of $d$ from being an harmonic map (i.e. is the stress-energy tensor of $d$), then
 we can calculate that
\begin{align*}\label{eqn: F(d)2}  
\begin{split}
-\curl\div \, \Xi & =- \partial_1\invHL(\partial_2 d\cdot \tau(d) ) + \partial_2 \invHL (\partial_1 d \cdot \tau(d)) \\
& = (\nabla \cdot \invHL)(\nabla^\perp d \cdot \tau(d) ),
\end{split} \end{align*}
which leads to the following lemma thanks to the standard boundedness in $L^p$ of the Riesz transform 

\begin{lemma}
The forcing term $F_1(d,D^\beta d)$ satisfies: for every $p \in (1,\infty)$, there exists $C>0$ such that 
$$
 \normp{F_1(d,D^\beta d)}{p} \leq C \normp{|\nabla d||\Delta d|}{p} 
 $$
\end{lemma}

The terms $J_1$ and $J_2$ can be then further simplified. A straightforward integration by parts gives for the first term 

\begin{align*}
J_1(t,x) & = -\int_{s\in(0,t)} \int_{y\in\bbR^2}  p_a(t-s,x-y) (u\cdot\nabla\theta)(s,y) \\
& = -\int_{s\in(0,t)} \int_{y\in\bbR^2}  \nabla p_a(t-s,x-y)\cdot (u\theta)(s,y)  .
\end{align*}
The algebraic expression \eqref{eqn: F(d)} allows us to rewrite the second term as

\begin{align*}
J_2(t,x)  & = \int_{s\in(0,t)} \int_{y\in\bbR^2}  p_a(t-s,x-y)\curl\div \, \Xi (s,y) \\
& = \int_{s\in(0,t)} \int_{y\in\bbR^2} (\partial_{11}-\partial_{22})\invHL  p_a(t-s,x-y) (\partial_2 d_j\partial_1 d_j)(s,y)  \\
&- \int_{s\in(0,t)} \int_{y\in\bbR^2}\partial_{12}\invHL p_a(t-s,x-y) (\partial_1d_j\partial_1d_j - \partial_2d_j\partial_2d_j)(s,y) \\
&= \int_0^t \int_{\bbR^2} Lp_a(t-s,x-y): A(\nabla d)(s,y) \,ds\,dy
\end{align*}
with 
\[ Lp_a = \begin{bmatrix} \partial_{x_1} \calR_1 p_a & \partial_{x_2}\calR_1p_a \\ \partial_{x_1} \calR_2 p_a& \partial_{x_2} \calR_2p_a \end{bmatrix} \]
and 
\[ A(\nabla d) = \begin{bmatrix} \partial_{x_1}d_j \partial_{x_2} d_j & - \partial_{x_1} d_j \partial_{x_1} d_j \\  \partial_{x_2} d_j \partial_{x_2} d_j &  -\partial_{x_1}d_j \partial_{x_2} d_j   \end{bmatrix} \]

The term  $Lp_a$ is a first order pseudo-differential operator, in the form of partial derivative of Riesz transform in spatial variable, applied to the heat kernel and the matrix  $A(\nabla d)$ is quadratic in $\nabla d$.

We apply \eqref{Est: HeatKernel} to estimate that 
\begin{align*} 
J_1(t,x)  & \leq \int_{s\in(0,t)} \int_{y\in\bbR^2}  |\nabla p_a (t-s,x-y)|  |u\theta|(s,y) \\
& \lesssim  \int_{s\in(0,t)} \int_{y\in\bbR^2} \frac{|u\theta|(s,y)}{\delta((x,t),(y,s))^3}=\pRieszP{2a-1}{|u\theta|}(t,x),
\end{align*}

\begin{align*} 
J_2(t,x) & \lesssim   \int_{s\in(0,t)} \int_{y\in\bbR^2} |\calR\nabla p_a (t-s,x-y)|  |\nabla d(s,y)|^2 \\
& \lesssim \int_{s\in(0,t)} \int_{y\in \bbR^2} \frac{|\nabla d(s,y)|^2}{\delta((x,t),(y,s))^3} =\pRieszP{2a-1}{|\nabla d|^2}(t,x).
\end{align*}

The $P_t\theta_0$ term can be estimated by the following inequality in \cite{WuEJDE01}
$$ \normp{P_t f}{p} \lesssim t^{-\frac{2a-1}{2a}+\frac{1}{ap}}\normp{f}{\frac{2}{2a-1}}, p\in\sbracket{\frac{2}{2a-1}, \infty } $$
which gives
\[\normp{P_t\theta_0}{L^p_{t,x}} \lesssim \frac{T^{\frac{2a-1}{2a}+\frac1p(1-\frac1a)}}{(1-\frac1a+\frac{2a-1}{2a}p)^{\frac1p}} \normp{\theta_0}{L_x^{\frac{2a}{2a-1}}} \leq C(p,a,T) \normp{\theta_0}{L_x^{\frac{2a}{2a-1}}}, \]
and $C(p,a,T)$ stays bounded if we fix $a,T$ and only let $p\to\infty$.

We then would like to apply Lemma \ref{Lem: Riesz potential1} below. We have assumption $\theta\in L^2\cap L^{p_0}((0,T)\times\bbR^2)$. Suppose that we have shown that $\theta\in L^{p_i}((0,T)\times\bbR^2)$, property of Riesz transform implies that we also have $u\in L^2\cap L^{p_i}((0,T)\times\bbR^2)$ and hence $|\theta u|\in L^1\cap L^{p_i/2}((0,T)\times \bbR^2)$.
For term $J_1(t,x)$, we can take $\lambda = 2+2a$ and $\beta=2a-1$. If $p_i<\frac{4+4a}{2a-1}$ then  Lemma \ref{Lem: Riesz potential1} implies that $|J_1|\lesssim \pRieszP{2a-1}{|\theta u|} \in L^{p_{1,i+1}}((0,T)\times\bbR^2)$ with $p_{1,i+1}=\frac{(2+2a)p_i}{4+4a-(2a-1)p_i}$.  If instead $p_i\geq \frac{4+4a}{2a-1}$, this is better because we can use interpolation to pick some $p\in(2,p_i)$ with $p<\frac{4+4a}{2a-1}$, and then obtain   $|J_1|\lesssim \pRieszP{2a-1}{|\theta u|} \in L^{p_{1,i+1}((0,T)\times\bbR^2)}$ with $p_{1,i+1}=\frac{(2+2a)p}{4+4a-(2a-1)p}$. By taking $p\to\frac{4+4a}{2a-1}$ we can show that $p_{1,i+1}$ can be arbitrarily large. Overall, our assumption that $\theta\in L^{p_i}((0,T)\times\bbR^2)$ with $p_i>\frac{2a+2}{2a-1}$ implies that $J_1\in L^{p_{1,i+1}}((0,T)\times\bbR^2)$ has improved regularity  $p_{1,i+1}>p_i$.

The regularity of $J_2$ relies on the regularity of $|\nabla d|^2$. We have assumption $\nabla d\in L^{q_0}((0,T)\times\bbR^2)$ with $q_0>4$. Suppose that we have $\nabla d\in L^{q_i}((0,T)\times\bbR^2)$, then we take $\lambda=4$ and $\beta=1$ in Lemma \ref{Lem: Riesz potential1}. Similar to argument in previous paragraph, if $q_i<\frac{4a+4}{2a-1}$, we  conclude that $|J_2|\lesssim \pRieszP{2a-1}{|\nabla d|^2}\in L^{p_{2,i+1}}((0,T)\times\bbR^2)$ with $p_{2,i+1}=\frac{(2+2a)q_i}{4+4a-(2a-1)q_i}$. If $q_i\geq\frac{4a+4}{2a-1}$ we conclude that $J_2\in L^{p_{2,i+1}}((0,T)\times\bbR^2)$.

In summary, from $\theta,u\in L^{p_i}((0,T)\times\bbR^2)$ and  $\nabla d\in L^{q_i}((0,T)\times\bbR^2)$ we can obtain $\theta \in L^{p_{i+1}}((0,T)\times\bbR^2)$ where 
\begin{equation}\label{Eqn: p_i} p_{i+1} =
\begin{cases} \frac{(2+2a)\min(p_i,q_i)}{4+4a-(2a-1)\min(p_i,q_i)} & \text{ if } \min(p_i,q_i)<\frac{4a+4}{2a-1}, \\
  \text{  arbitrarily large } & \text{ if }  \min(p_i,q_i)\geq \frac{4a+4}{2a-1},
\end{cases} \end{equation}
and $p_{i+1}>\min(p_i,q_i)$ if $\min(p_i,q_i)>\frac{2a+2}{2a-1}$.

 We similarly use Duhamel principle to write mild solution formula for $\eqref{eqn: SQGHMF}_3$ 
as 
$$ d(t,x) = P_{1,t}d_0(x) + \int_{s\in(0,t)}\int_{y\in\bbR^2} p_1(t-s,x-y) G(s,y), $$
where
$$ G = |\nabla d|^2 d - u\cdot \nabla d \ \  \in L^1\cap L^{\frac12\min(p_i,q_i)}((0,T)\times\bbR^2).$$
As a consequence, 
\begin{align*}
 |\nabla d(t,x)| & \leq  |P_{1,t}\nabla d_0(x)| + \int_{s\in(0,t)}\int_{y\in\bbR^2} |\nabla p_1(t-s,x-y)| |G(s,y)| \\
 & \leq  |P_{1,t}\nabla d_0(x)| + \int_{s\in(0,t)}\int_{y\in\bbR^2} \frac{ |G(s,y)|}{\max(|x-y|,|t-s|^{\frac12})^3}.
  \end{align*}
We take $\beta=1,\lambda=4$ in Lemma \ref{Lem: Riesz potential1} and conclude that $\nabla d(t,x)\in L^{q_{i+1}}((0,T)\times\bbR^2)$ where 
\begin{equation}\label{Eqn: q_i} q_{i+1}=\begin{cases} 
\frac{4\min(p_i,q_i)}{8-\min(p_i,q_i)} & \text{ if } \min(p_i,q_i)<8, \\
\text{arbitrarily large} &  \text{ if } \min(p_i,q_i)\geq 8.
\end{cases}\end{equation}
Similarly, we can see directly that $q_{i+1}>\min(p_i,q_i)$ if $\min(p_i,q_i)>4$.

Now we apply bootstrap technique to tuple $(p_i,q_i)$. Starting from $(p_i,q_i)$ with $p_i>\frac{2a+2}{2a-1}>4$ and $q_i>4$, we first fix $p_{i+1}=p_i$ unchanged, and merely iterate over \eqref{Eqn: q_i} to finally improve $q_{i+1}$ to magnitude $q_{i+1}= p_{i+1}= p_i>\frac{2a+2}{2a-1}$, then we have $\min(p_{i+1},q_{i+1})>\frac{2a+2}{2a-1}$. Second, we fix $q_{i+2}=q_{i+1}$ unchanged and apply \eqref{Eqn: p_i} to improve $p_{i+2}$ to magnitude $p_{i+2}>\min(p_{i+1},q_{i+1})=p_{i+1}=p_i$. Third, we fix $p_{i+3}=p_{i+2}>p_{i+1}$ unchanged and apply \eqref{Eqn: q_i} to  improve $q_{i+3}>\min(p_{i+2},q_{i+2})\geq q_{i+1}$. Therefore, we have both $p_{i+3}>p_{i+1}$ and $q_{i+3}>q_{i+1}$, and we iterate over above argument to conclude that $\theta,u,\nabla d\in L^p((0,T)\times\bbR^2)$ for arbitrarily large $p$.

\subsection{H\"older regularity: proof of Theorem \ref{Prop: Reg2-intro}}

First,  Theorem \ref{Prop: Reg1-intro} states that $\theta,u,\nabla d\in L^p((0,T)\times\bbR^2)$ for arbitrarily large $p\in[2,\infty)$. To prove H\"older regularity, we can see in this proof that we need $|\theta u|,|\nabla d|^2 \in L^p((0,T)\times\bbR^2)$ for $p>\frac{2a+2}{2a-1}$. We set
$$ 
\theta(z_1) =P_{t_1}\theta_0(x_1) + J_1(t_1,x_1)+ J_2(t_1,x_1),
$$
where
\begin{align*}
J_1(t,x) &= \int_0^t \int_{\bbR^2} \nabla p_a(t-s,x-y) \cdot (\theta u)(s,y) dsdy, \\
J_2(t,x) &= \int_0^t \int_{\bbR^2} Lp_a(t-s,x-y): A(\nabla d)(s,y) dsdy
\end{align*}
with the same notations as above. 
 Because $P_t\theta_0(x)$ denotes the solution of fractional heat equation $\partial_t\theta + \FL{a}\theta=0$ with initial condition $\theta(0,x)=\theta_0(x)$, then one can show that $P_t\theta_0$ is smooth on $(0,T)\times\bbR^2$. Indeed, for any $b_0,b_1,b_2\in\bbN$,
\begin{align*}
 |\partial_t^{b_0} \partial_{x_1}^{b_1} \partial_{x_2}^{b_2} P_t\theta_0|(t,x) \leq &~ \int_{\bbR^2} |\partial_t^{b_0} \partial_{x_1}^{b_1} \partial_{x_2}^{b_2} p(t,x-y) | |\theta_0(y)| dy \\ 
\lesssim &~ t^{-(b_0+\frac{2+b_1+b_2}{2a})} \normp{\theta_0}{L^1(\bbR^2)}.
\end{align*}

To show the H\"older regularity for $J_1$ and $J_2$ term, we show that they are in some fractional parabolic Campanato space $\Campanato{1,\lambda}{Q_r(t_0,x_0)}$. Formally, if one treat formally $\nabla p(z)$ and $Lp(z)$ as $\delta(0,z)^{-3}$, then formally $J_1,J_2$ are $\pRieszP{2a-1}{|u\theta|},\pRieszP{2a-1}{|\nabla d|^2}$, and also Lemma \ref{Lem: Riesz potential2} would imply H\"older continuity with exponent $$2a-1-\frac{2+2a}{p}=\frac{2a+2}{p'}-3.$$ Here we will follow similar approach but the estimation is more delicate in that $\nabla p(z)$ and $Lp(z)$ differ from $\delta(z)^{-3}$ in two ways: different finite support in time and different finite difference estimate \eqref{Eqn: Tech1}.

To estimate $J_1$, observe that 
\begin{align*}
 &~\dashint_{(t_1,x_1)\in Q_r(t_0,x_0)}\abs{J_1(t_1,x_1) - (J_1)_{Q_r(t_0,x_0)}} \\
\leq&~ \dashint_{(t_1,x_1)\in Q_r(t_0,x_0)} \dashint_{(t_2,x_2)\in Q_r(t_0,x_0)} \abs{J_1(t_1,x_1)-J_1(t_2,x_2)} \\
= &~  \dashint_{\substack{(t_1,x_1)\in \\ Q_r(t_0,x_0)}} \dashint_{\substack{(t_2,x_2)\in \\ Q_r(t_0,x_0)}} \bigg|\int_{\substack{s\in(0,t_1) \\ y\in\bbR^2}} \theta u(s,y) \cdot \nabla p(t_1-s,x_1-y)  \\
&~\qquad\qquad\qquad\qquad\quad- \int_{\substack{s\in(0,t_2) \\ y\in\bbR^2}} \theta u(s,y) \cdot \nabla p(t_2-s,x_2-y) \bigg| \\
 \lesssim &~ J_{11} + J_{12} + J_{13},
\end{align*}
where 
\begin{align*} 
J_{11} &= \dashint_{\substack{(t_1,x_1)\in \\ Q_r(t_0,x_0)}} \dashint_{\substack{(t_2,x_2)\in \\ Q_r(t_0,x_0) }} \II\{t_2<t_1\} \int_{\substack{s\in(t_2,t_1) \\ y\in\bbR^2}}  \abs{ \theta u(s,y)} \abs{ \nabla p(t_1-s,x_1-y) } \\
J_{12} & =  \dashint_{\substack{(t_1,x_1)\in \\ Q_r(t_0,x_0)}} \dashint_{\substack{(t_2,x_2)\in \\ Q_r(t_0,x_0) }}  \int_{\substack{s\in(0,t_2) \\ y\in\bbR^2}} \abs{ \theta u(s,y)} \abs{ \nabla p(t_1-s,x_1-y) - \nabla p(t_1-s,x_2-y) } \\
J_{13} & = \dashint_{\substack{(t_1,x_1)\in \\ Q_r(t_0,x_0)}} \dashint_{\substack{(t_2,x_2)\in \\ Q_r(t_0,x_0) }} \II\{t_2<t_1\}  \int_{\substack{s\in(0,t_2) \\ y\in\bbR^2}} \abs{ \theta u(s,y)} \\
&~\qquad\qquad\qquad\qquad\qquad\qquad\qquad\qquad\abs{ \nabla p(t_1-s,x_1-y) - \nabla p(t_2-s,x_1-y) },
\end{align*}
and for $p>\frac{2+2a}{2a-1}$, i.e. $p'\in(1,\frac{2a+2}{3})$, can estimate $J_{11},J_{12},J_{13}$ in the following way. Notice that we will use the trick of separating integral into spatial ``inner" and ``outer" parts in Lemma \ref{Lem: Riesz potential2}. For  inner part we use $\abs{\nabla p(z)}\lesssim \delta(0,z)^{-3}$. For outer part, for term $J_{11}$  we first integrate over spatial variable $y$ and then integrate over time variable $s$. For term $J_{12}$  which has finite difference of $\nabla p$ in spatial variable $x$, we introduce a dummy variable $h\in[0,1]$ and write the finite difference in terms of  $\sup_{h\in[0,1]}\frac{d\nabla p}{d h}$ and use \eqref{Est: HeatKernel} to control derivatives. Details are as follows:
\begin{align*}
J_{11}  \lesssim & \sup_{\substack{(t_1,x_1)\in \\ Q_r(t_0,x_0)}} \int_{\substack{\Delta t := t_1-t_2 \\ \in (0,2r^{2a})}}  \rbracket{\int_{\substack{s\in(0,\Delta t) \\ |y|\leq |\Delta t|^{\frac1{2a}}}} + \int_{\substack{s\in(0,\Delta t) \\ |y|> |\Delta t|^{\frac1{2a}}}}} \abs{\theta u (t_1-s,x_1-y)} \abs{\nabla p(s,y)} \\
 \lesssim & \sup_{\substack{(t_1,x_1)\in \\ Q_r(t_0,x_0)}} \int_{\substack{\Delta t := t_1-t_2 \\ \in (0,2r^{2a})}} \normp{\theta u}{L^p\rbracket{Q_{|\Delta t|^{1/2a}}(t_1,x_1)}} \normp{\delta(0,\cdot)^{-3}}{L^{p'}\rbracket{Q_{|\Delta t|^{1/2a}}(t_1,x_1)}} \\
& +  \sup_{\substack{(t_1,x_1)\in \\ Q_r(t_0,x_0)}} \int_{\substack{\Delta t := t_1-t_2 \\ \in (0,2r^{2a})}} \int_{s\in(0,\Delta t)} \normp{\theta u(t_1-s,\cdot)}{L^p(\bbR^2)} \normp{|y|^{-3}}{|y|> |\Delta t|^{1/2a}} \\
\lesssim &  \sup_{\substack{(t_1,x_1)\in \\ Q_r(t_0,x_0)}} \int_{\substack{\Delta t := t_1-t_2 \\ \in (0,2r^{2a})}} \normp{\theta u}{L^p\rbracket{Q_{|\Delta t|^{1/2a}}(t_1,x_1)}} |\Delta t|^{\frac{1}{2a}(\frac{2a+2}{p'}-3)} \\
& +  \sup_{\substack{(t_1,x_1)\in \\ Q_r(t_0,x_0)}} \int_{\substack{\Delta t := t_1-t_2 \\ \in (0,2r^{2a})}} \int_{s\in(0,\Delta t)}  \normp{\theta u(t_1-s,\cdot)}{L^p(\bbR^2)} |\Delta t|^{\frac{1}{2a}(\frac{2}{p'}-3)} \\
\lesssim & \int_{\Delta t \in (0,2r^{2a})} \normp{\theta u}{L^p((0,T)\times \bbR^2)} |\Delta t|^{\frac1{2a}(\frac{2a+2}{p'}-3)} \lesssim \normp{\theta u}{L^p((0,T)\times \bbR^2)} r^{\frac{2a+2}{p'}-3},
\end{align*}

\begin{align*}
J_{12} 
\lesssim &\sup_{\substack{(t_1,x_1)\in \\ Q_r(t_0,x_0)}} \dashint_{\substack{\Delta x := x_2-x_1 \\ \in B(0,2r)}}   \int_{\substack{\delta(0,(s,y))\\ \leq 2|\Delta x|}}  \abs{\theta u(t_1-s,x_1-y)} \abs{\nabla p(s,y) -\nabla p(s,y+\Delta x)} \\
&+\sup_{\substack{(t_1,x_1)\in \\ Q_r(t_0,x_0)}} \dashint_{\substack{\Delta x := x_2-x_1 \\ \in B(0,2r)}}   \int_{\substack{\delta(0,(s,y))\\>2|\Delta x|}} \abs{\theta u(t_1-s,x_1-y)} \abs{\nabla p(s,y) -\nabla p(s,y+\Delta x)} \\
\lesssim &\sup_{\substack{(t_1,x_1)\in \\ Q_r(t_0,x_0)}} \dashint_{\substack{\Delta x := x_2-x_1 \\ \in B(0,2r)}}   \int_{\substack{\delta(0,(s,y))\\ \leq 3|\Delta x|}}  \abs{\theta u(t_1-s,x_1-y)} \abs{\nabla p(s,y) } \\
&+\sup_{\substack{(t_1,x_1)\in \\ Q_r(t_0,x_0)}} \dashint_{\substack{\Delta x := x_2-x_1 \\ \in B(0,2r)}}   \int_{\substack{\delta(0,(s,y))\\>2|\Delta x|}} \abs{\theta u(t_1-s,x_1-y)} \abs{\sup_{h\in[0,1]}\frac{d}{d h}\nabla p(s,y+h\Delta x)} \\
\lesssim &\sup_{\substack{(t_1,x_1)\in \\ Q_r(t_0,x_0)}} \dashint_{\substack{\Delta x := x_2-x_1 \\ \in B(0,2r)}}  
\normp{\theta u}{L^p\rbracket{Q_{3|\Delta x|}(t_1,x_1)}} \normp{\delta(0,\cdot)^{-3}}{L^{p'}\rbracket{Q_{3|\Delta x|}(0)}}\\
&+\sup_{\substack{(t_1,x_1)\in \\ Q_r(t_0,x_0)}} \dashint_{\substack{\Delta x := x_2-x_1 \\ \in B(0,2r)}}   \normp{\theta u}{L^p(0,T\times\bbR^2)} |\Delta x| \normp{\delta(0,\cdot)^{-4}}{L^{p'}\rbracket{\delta(0,z)>2|\Delta x|}} \\
\lesssim &\sup_{\substack{(t_1,x_1)\in \\ Q_r(t_0,x_0)}} \dashint_{\substack{\Delta x := x_2-x_1 \\ \in B(0,2r)}}  
\normp{\theta u}{L^p(0,T\times\bbR^2)} |\Delta x|^{\frac{2a+2}{p'}-3} \lesssim  \normp{\theta u}{L^p((0,T)\times\bbR^2)} r^{\frac{2a+2}{p'}-3}.
\end{align*}
For term $J_{13}$, we instead separate into ``inner" and ``outer" terms with respect to time, i.e. $s\in(0,2\Delta t)$ and $s\in(2\Delta t,t_1)$.  
\begin{align*}
J_{13}  \lesssim  &  \sup_{\substack{(t_1,x_1)\in \\ Q_r(t_0,x_0)}} \int_{\substack{\Delta t := t_1-t_2 \\ \in (0,2r^{2a})}} \int_{\substack{s\in(0,2\Delta t) \\ y\in\bbR^2}} \abs{\theta u (t_2-s,x_1-y)} \abs{\nabla p(s,y) - \nabla p(s+\Delta t,y)} \\
& + \sup_{\substack{(t_1,x_1)\in \\ Q_r(t_0,x_0)}} \int_{\substack{\Delta t := t_1-t_2 \\ \in (0,2r^{2a})}} \int_{\substack{s\in(2\Delta t,t_1) \\ y\in\bbR^2}} \abs{\theta u (t_2-s,x_1-y)} \abs{\nabla p(s,y) - \nabla p(s+\Delta t,y)}.
\end{align*}
Notice from above that the first term (i.e. ``inner-in-time" part) can be controlled in the same way as  $J_{11}$, so we simplify our computation and obtain
\begin{align*}  \sup_{\substack{(t_1,x_1)\in \\ Q_r(t_0,x_0)}} \int_{\substack{\Delta t := t_1-t_2 \\ \in (0,2r^{2a})}} \int_{\substack{s\in(0,2\Delta t) \\ y\in\bbR^2}} \abs{\theta u (t_2-s,x_1-y)} \abs{\nabla p(s,y) - \nabla p(s+\Delta t,y)}  \lesssim  \\
s \normp{\theta u}{L^p((0,T)\times\bbR^2)} r^{\frac{2a+2}{p'}-3}  
\end{align*}
and the second ``outer-in-time" term is estimated as follows:
\begin{align*}
& \sup_{\substack{(t_1,x_1)\in \\ Q_r(t_0,x_0)}} \int_{\substack{\Delta t := t_1-t_2 \\ \in (0,2r^{2a})}} \int_{\substack{s\in(2\Delta t,t_1) \\ y\in\bbR^2}} \abs{\theta u (t_2-s,x_1-y)} \abs{\nabla p(s,y) - \nabla p(s+\Delta t,y)} \\
\lesssim & \sup_{\substack{(t_1,x_1)\in \\ Q_r(t_0,x_0)}} \int_{\substack{\Delta t := t_1-t_2 \\ \in (0,2r^{2a})}} \int_{\substack{s\in(2\Delta t,t_1) \\ y\in\bbR^2}} \abs{\theta u (t_2-s,x_1-y)} \abs{\sup_{h\in[0,1]}\frac{d}{dh}\nabla p(s+h\Delta t,y)} \\
 \lesssim & \sup_{\substack{(t_1,x_1)\in \\ Q_r(t_0,x_0)}} \int_{\substack{\Delta t := t_1-t_2 \\ \in (0,2r^{2a})}} \rbracket{\int_{\substack{s\in(2\Delta t,t_1) \\ |y|\leq s^{1/2a}}}+ \int_{\substack{s\in(2\Delta t,t_1) \\ |y|>s^{1/2a}}}}\frac{\abs{\theta u (t_2-s,x_1-y)}}{\delta(0,(s,y))^3}\frac{\Delta t}{s} \\
\lesssim &  \sup_{\substack{(t_1,x_1)\in \\ Q_r(t_0,x_0)}} \int_{\substack{\Delta t := t_1-t_2 \\ \in (0,2r^{2a})}} \int_{s\in(2\Delta t,t_1)} |\Delta t| \normp{\theta u(t_2-s,\cdot)}{L^p(\bbR^2)} s^{\frac1{ap'}-\frac{3}{2a}-1} \\
\lesssim &  \sup_{\substack{(t_1,x_1)\in \\ Q_r(t_0,x_0)}} \int_{\substack{\Delta t := t_1-t_2 \\ \in (0,2r^{2a})}}   \normp{\theta u}{L^p((0,T)\times\bbR^2)} |\Delta t|^{\frac{a+1}{ap'}-\frac{3}{2a}} \\
\lesssim & \normp{\theta u}{L^p((0,T)\times\bbR^2)} r^{\frac{2a+2}{p'}-3}  .
\end{align*}
This finishes the proof that $J_1 \in \Campanato{1,3-\frac{2a+2}{p'}}{(0,T)\times\bbR^2} = C^{3-\frac{2a+2}{p'}}_{loc,\delta}((0,T)\times\bbR^2)$. The proof of H\"older regularity for $J_2$ is similar, merely  replacing $\theta u$ by $A(\nabla d)$ and replacing $\nabla p$ by $Lp$.  

\subsection{Proof of the main Theorem \ref{Thm: Reg-intro}}

First by Theorems \ref{Prop: Reg1-intro} and  \ref{Prop: Reg2-intro} we conclude that $\theta,~u,~\nabla d\in L^p((0,T)\times\bbR^2)$ for any $p\geq 2$ and $\theta,~u\in C^\alpha_{loc}((0,T)\times\bbR^2)$ for some $\alpha>0$. We would like to show H\"older regularity of $\nabla d$ as well. Consider
$$ \partial_t d -\Delta d = -u\cdot \nabla d + |\nabla d|^2 d .$$
We use $W^{2,1}_p$-estimate for heat equation to show that $d\in W^{2,1}_p(P_r(z))$ for any parabolic cylinder $P_r(z)\subseteq (0,T)\times\bbR^2$ and any $p\geq 2$. As a consequence, Sobolev inequality shows that $\nabla d\in C^\alpha(P_{r}(z))$ for some $\alpha>0$.

As a consequence, Schauder estimate gives $d\in C^{2,1,\alpha}_{loc}((0,T)\times\bbR^2)$.   Expression \eqref{eqn: F(d)} for $F(d)$ gives $F(d)\in C^\alpha_{loc}((0,T)\times\bbR^2)$. Then we follow the same argument as in Theorem \ref{Prop: Reg2-intro} on second-order finite difference $\theta(z+h_0)+\theta(z-h_0)-2\theta(z) $ to show that that $\theta,u\in C^{1,\alpha}_{loc}((0,T)\times\bbR^2)$. After that we can bootstrap to smoothness.

\section*{Appendix: Parabolic potentials and Morrey spaces}

We collect in this section several intermediate results, instrumental towards the proof. Most of the results are already in the literature but we provide an argument for those which are particular to our framework. 

\begin{definition}
We denote $P_r(x,t)$ to be parabolic cylinder
$$ P_r(x,t) = \cbracket{ (x',t'): \max(|x-x'|,|t-t'|^{\frac12}) < r }, $$
and let $Q_r(x,t)$ denote the fractional parabolic cylinder
$$ Q_r(x,t) = \cbracket{ (x',t'): \max(|x-x'|, |t-t'|^{\frac{1}{2a}})<r } .$$
We use $\delta$ to denote fractional parabolic metric
$$ \delta((x,t),(x',t')) =   \max(|x-x'|, |t-t'|^{\frac{1}{2a}}). $$
Also we use $\Omega_T$ to denote $\Omega\times(0,T)$ where $\Omega\subseteq\bbR^2$.
\end{definition}



\begin{lemma}[Fractional Parabolic Poincare Inequality] Let fractional parabolic cylinder be $Q_r(x,t)=\cbracket{(y,s)\in\bbR^{n+1}: |x-y|,|t-s|^{\frac1{2a}} \leq r} = B_r \times T_r $ where $B_r =\cbracket{ |x-y|\leq r }$ and $T_r = \cbracket{|t-s|\leq r^{2a}} $. We claim that
\[ \dashint_{Q_r} |\theta - \theta_{Q_r}|^p \lesssim \dashint_{Q_r} r^{2\beta p}|\FL{\beta}\theta|^p + r^{2ap}|\partial_t\theta|^p. \]
\end{lemma}
\begin{proof}
\begin{align*} LHS & =\dashint_{(x,t)\in Q_r} \abs{ \dashint_{s\in T_r} \theta(t,x) -\dashint_{y\in B_r}\theta(s,y) }^p  \\
& \leq   \dashint_{(x,t)\in Q_r} \dashint_{s\in T_r} \abs{  \theta(t,x) -\dashint_{y\in B_r}\theta(s,y) }^p \\
& \lesssim  \dashint_{(x,t)\in Q_r} \dashint_{s\in T_r} \abs{  \theta(t,x) -\theta(s,x)}^p +\abs{\theta(s,x) -\dashint_{y\in B_r}\theta(s,y) }^p,
\end{align*}
where the first term  
\begin{align*}
 \dashint_{x\in B_r}\dashint_{t\in T_r} \dashint_{s\in T_r} \frac{\abs{  \theta(t,x) -\theta(s,x)}^p }{|t-s|^{1+p}} |t-s|^{1+p} \lesssim&~   \dashint_{x\in B_r} |T_r|^{p-1} \normp{\partial_t\theta}{L^p(T_r)}^p\\
  \lesssim&~ \dashint_{Q_r} |\partial_t\theta|^p r^{2ap},
\end{align*}
and the second term 
\begin{align*}
 \dashint_{s\in T_r} \dashint_{ x\in B_r} \abs{\theta(s,x) -\dashint_{y\in B_r}\theta(s,y) }^p \leq &~\dashint_{s\in T_r} |B_r|^{p(2\beta)/2-1} \seminormp{\theta(s,\cdot)}{W^{2\beta,p}_x}^p \\
  \lesssim&~ \dashint_{Q_r} r^{2\beta p} \abs{\FL{\beta}\theta}^p. 
 \end{align*}
\end{proof}


We will make use of potential theory to derive the regularity and we will use the following Campanato-Morrey spaces. 

\begin{definition}[Parabolic Morrey space, Parabolic  Campanato space, Maximal operator]
Let $L^{p,\lambda}(\Omega_T)$ and $L^{p,\lambda}_*(\Omega_T)$, respectively denote the parabolic Morrey space, and the weak parabolic Morrey space, where $1\leq p<\infty, 0\leq\lambda\leq n+2$, and 
\begin{equation*}\begin{split}
\normp{f}{L^{p,\lambda}(\Omega_T)}^p &:= \sup_{r>0,z\in\Omega_T} r^{\lambda - (n+2)}\normp{f}{L^p(\Omega_T\cap P_r(z))}^p, \\
\normp{f}{L^{p,\lambda}_*(\Omega_T)}^p &:= \sup_{r>0,z\in\Omega_T} r^{\lambda - (n+2)}\normp{f}{L^{p,*}(\Omega_T\cap P_r(z))}^p .
\end{split}\end{equation*}

Let $\scrL^{p,\lambda}(\Omega_T)$ and $\scrL^{p,\lambda}_*(\Omega_T)$, respectively denote the  parabolic Campanato space, and the weak  parabolic Campanato space, where $1\leq p<\infty, -p<\lambda< n+2$, and 
\begin{equation*}\begin{split}
\normp{f}{\scrL^{p,\lambda}(\Omega_T)}^p& := \sup_{r>0,z\in\Omega_T} r^{\lambda - (n+2)}\normp{f-(f)_{\Omega_T\cap P_r(z)}}{L^p(\Omega_T\cap P_r(z))}^p, \\
\normp{f}{\scrL^{p,\lambda}_*(\Omega_T)}^p& := \sup_{r>0,z\in\Omega_T} r^{\lambda - (n+2)}\normp{f-(f)_{\Omega_T\cap P_r(z)}}{L^{p,*}(\Omega_T\cap P_r(z))}^p ,
\end{split}\end{equation*}
where $(f)_{\Omega_T\cap P_r(z)}$ denotes the average value of $f$ on domain $\Omega_T\cap P_r(z)$:
$$ (f)_{\Omega_T\cap P_r(z)} = \frac{1}{|\Omega_T\cap P_r(z)|} \int_{\Omega_T\cap P_r(z)} f(z') dz' .$$
The maximal operator $M_\alpha f, 0<\alpha<n+2a$, is defined as follows:
$$ M_\alpha f(z) = \sup_{r>0} r^\alpha \dashint_{z'=(t',x')\in P_r(z)} |f(z')| dz', \ \ z=(t,x) .$$
\end{definition}

\begin{definition}[Fractional Parabolic Morrey space, Fractional Parabolic  Campanato space, Maximal operator]
Let $\Morrey{p,\lambda}{\Omega_T}$ and $\wMorrey{p,\lambda}{\Omega_T}$, respectively denote the (fractional) parabolic Morrey space, and the weak fractional parabolic Morrey space, where $1\leq p<\infty, 0\leq\lambda\leq n+2a$, and 
\begin{equation*}\begin{split}
\normp{f}{\Morrey{p,\lambda}{\Omega_T}}^p &:= \sup_{r>0,z\in\Omega_T} r^{\lambda - (n+2a)}\normp{f}{L^p(\Omega_T\cap Q_r(z))}^p, \\
\normp{f}{\wMorrey{p,\lambda}{\Omega_T}}^p &:= \sup_{r>0,z\in\Omega_T} r^{\lambda - (n+2a)}\normp{f}{L^{p,*}(\Omega_T\cap Q_r(z))}^p .
\end{split}\end{equation*}

Let $\Campanato{p,\lambda}{\Omega_T}$ and $\wCampanato{p,\lambda}{\Omega_T}$, respectively denote the fractional parabolic Campanato space, and the weak fractional parabolic Campanato space, where $1\leq p<\infty, -p<\lambda< n+2a$, and 
\begin{equation*}\begin{split}
\normp{f}{\Campanato{p,\lambda}{\Omega_T}}^p& := \sup_{r>0,z\in\Omega_T} r^{\lambda - (n+2a)}\normp{f-(f)_{\Omega_T\cap Q_r(z)}}{L^p(\Omega_T\cap Q_r(z))}^p, \\
\normp{f}{\wCampanato{p,\lambda}{\Omega_T}}^p& := \sup_{r>0,z\in\Omega_T} r^{\lambda - (n+2a)}\normp{f-(f)_{\Omega_T\cap Q_r(z)}}{L^{p,*}(\Omega_T\cap Q_r(z))}^p ,
\end{split}\end{equation*}
where $(f)_{\Omega_T\cap Q_r(z)}$ denotes the average value of $f$ on domain $\Omega_T\cap Q_r(z)$:
$$ (f)_{\Omega_T\cap Q_r(z)} = \frac{1}{|\Omega_T\cap Q_r(z)|} \int_{\Omega_T\cap Q_r(z)} f(z') dz'. $$
The maximal operator $\widetilde{M}_\alpha f, 0<\alpha<n+2a$, is defined as follows:
$$ \widetilde{M}_\alpha f(z) = \sup_{r>0} r^\alpha \dashint_{z'=(t',x')\in Q_r(z)} |f(z')| dz', \ \ z=(t,x) .$$
\end{definition}
\emph{Remarks:} Notice that the definition of fractional parabolic Morrey and Campanato space degenerates to parabolic Morrey and Campanato space as long as one take $a$ by $1$ and substitute $Q_r$ by $P_r$. Thus the following Lemmas on fractional parabolic Morrey and Campanato space all have  analogous versions in ordinary parabolic space, proved by the same argument.

\begin{lemma}[Identification of Fractional Parabolic  Campanato space]\label{Lem:  Campanato} For bounded domain $\Omega_T$, the fractional parabolic Campanato space $\Campanato{p,\lambda}{\Omega_T}$ can be identified as follows:
$$\Campanato{p,\lambda}{\Omega_T} = 
\begin{cases}  \Morrey{p,\lambda}{\Omega_T} & 0<\lambda<n+2a, \\  BMO(\Omega_T) & \lambda = 0, \\ C^{\alpha}_\delta(\Omega_T) & -p<\lambda<0, \alpha = -\frac{\lambda}{p}.
\end{cases} $$
\end{lemma}
\begin{proof}

The proof of this lemma uses the following technical result for $f\in \Campanato{p,\lambda}{\Omega_T}$:
\begin{equation}\label{Eqn: Tech2} \ \ \ \ \ \ \ \ \ \ \ \ \ 
\abs{(f)_{Q_{\rho 2^{-(i+1)}}(x_0)} - (f)_{Q_{\rho 2^{-i}}}(x_0) }  \leq C(n,p,\lambda) \rho^{\frac{-\lambda}{p}}2^{i(\frac{\lambda}{p})} \normp{f}{\Campanato{p,\lambda}{\Omega_T}}.
\end{equation}

 Take $0<r_1<r_2\leq \diam(\Omega_T)$ and $x_0\in\Omega$. We can calculate that 
\begin{align*}
& \abs{(f)_{Q_{r_1}(x_0)} - (f)_{Q_{r_2}(x_0)} }^p \lesssim  \abs{f(x) - (f)_{Q_{r_1}(x_0)} }^p + \abs{f(x) - (f)_{Q_{r_2}(x_0)} }^p \\
& \lesssim \dashint_{Q_{r_1}(x_0)} \abs{f(x) - (f)_{Q_{r_1}(x_0)} }^p  + \dashint_{Q_{r_2}(x_0)} \abs{f(x) - (f)_{Q_{r_2}(x_0)} }^p \\
& \lesssim {r_1}^{-(n+2a)}(r_1^{n+2a-\lambda}  + r_2^{n+2a-\lambda}) \normp{f}{\Campanato{p,\lambda}{\Omega_T}}^p.
\end{align*}
Fix $0<\rho\leq 2\diam(\Omega_T)$ and we substitute $r_1=\rho 2^{-(i+1)}$ and $r_2=\rho 2^{-i}$ to finish the proof of \eqref{Eqn: Tech2}.  

Part 1 consists of proving that fractional parabolic Campanato space equals fractional parabolic Morrey space. $\Campanato{p,\lambda}{\Omega_T} = \Morrey{p,\lambda}{\Omega_T} $. One direction $f\in \Morrey{p,\lambda}{\Omega_T}\hookrightarrow \Campanato{p,\lambda}{\Omega_T}$ is direct because one can show that
$$ \normp{f}{\Campanato{p,\lambda}{\Omega}}^p = \sup_{r>0,z\in\Omega_T} r^{\lambda-(n+2a)} \inf_{c\in\bbR} \normp{f-c}{L^p(\Omega_T\cap Q_r(z))}^p \leq \normp{f}{\Morrey{p,\lambda}{\Omega}}^p .$$
To prove the other direction $f\in\Campanato{p,\lambda}{\Omega_T} \hookrightarrow \Morrey{p,\lambda}{\Omega_T}$, we fix $0<\rho<\diam(\Omega_T)$ and choose $k$ such that $\frac{\diam(\Omega_T)}{2^{k+1}} \leq \rho < \frac{\diam(\Omega_T)}{2^{k}} $, we can show that 
$$ \abs{(f)_{Q_\rho(x)}} \leq \abs{(f)_{\Omega_T}} +  \abs{(f)_{\Omega_T} - (f)_{Q_{ \diam(\Omega_T)2^{-k}}(x)}} + \abs{ (f)_{Q_{\diam(\Omega_T) 2^{-k}}(x)} - (f)_{Q_\rho(x)}}  .$$
By \eqref{Eqn: Tech2}, we can control the second and the third term on the right hand side by
\begin{align*}
 \abs{(f)_{\Omega_T} - (f)_{Q_{ \diam(\Omega_T)2^{-k}}(x)}} & \lesssim \normp{f}{\Campanato{p,\lambda}{\Omega_T}} \diam(\Omega_T)^{-\frac{\lambda}{p}} \sum_{i=0}^k 2^{i\frac{\lambda}{p}} \\
& \lesssim \normp{f}{\Campanato{p,\lambda}{\Omega_T}} \rho^{-\frac{\lambda}{p}} \frac{1- 2^{-k\frac{\lambda}{p}}}{2^{\frac{\lambda}{p}}-1} ,\\
\abs{ (f)_{Q_{\diam(\Omega_T) 2^{-k}}(x)} - (f)_{Q_\rho(x)}}  & \lesssim \rho^{-\frac{n+2a}{p}} \rbracket{\rho^{\frac{n+2a-\lambda}{p}}+ \rbracket{\frac{d}{2^k}}^{\frac{n+2a-\lambda}{p}}} \normp{f}{\Campanato{p,\lambda}{\Omega_T}} \\
& \lesssim  \normp{f}{\Campanato{p,\lambda}{\Omega_T}} \rho^{-\frac{\lambda}{p}} 2^{\frac{n+2a-\lambda}{p}},
\end{align*}
which implies that $\abs{(f)_{Q_\rho(x)}}\lesssim\normp{f}{\Campanato{p,\lambda}{\Omega_T}} \rho^{-\frac{\lambda}{p}} $ and hence $f\in \Morrey{p,\lambda}{\Omega_T}$.

Part 2 that $\Campanato{p,0}{\Omega_T} = BMO(\Omega_T)$ is obvious because their definition concides.

Part 3 consists of proving that $\Campanato{p,\lambda}{\Omega_T} = C^{-\frac{\lambda}{p}}_\delta(\Omega_T)$ when $-p<\lambda<0$. We follow the argument in \cite{MRTroianiello87}.
One direction that $f \in C^\alpha_\delta(\Omega_T) \hookrightarrow \Campanato{p,-p\alpha}{\Omega_T}$ is direct. Notice that  integral mean value theorem gives for $r>0$:
$$ \abs{ f(x) - (f)_{Q_r(x_0)\cap \Omega_T} }\leq [f]_{C^\alpha_\delta}^p 2^{p\alpha} r^{p\alpha}r^{n+2a} $$
and hence by definition $\normp{f}{\Campanato{p,-p\alpha}{\Omega_T}} \leq 2^\alpha [f]_{C^\alpha_\delta}\lesssim \normp{f}{C^\alpha_\delta(\Omega_T)}$.

Now we show that $f \in \Campanato{p,\lambda}{\Omega_T} \hookrightarrow C^{-\frac{\lambda}{p}}_\delta(\Omega_T)$.
We take $k_2>k_1>k_0$ in \eqref{Eqn: Tech2} and then we have 
$$ \abs{(f)_{Q_{\rho 2^{-k_2}}(x_0)} - (f)_{Q_{\rho 2^{-k_1}}(x_0)}} \lesssim \rho^{-\frac{\lambda}{p}}  \normp{f}{\Campanato{p,\lambda}{\Omega_T}} \sum_{i=k_0}^\infty 2^{i\frac{\lambda}{p}} $$
which converges to $0$ as $k_0\to\infty$, thus $\cbracket{(f)_{Q_{\rho 2^{-i}}(x_0)}}_{i=1}^\infty$ is a Cauchy sequence and hence converges. We redefine $f(x_0)$ by $\lim_{i\to\infty}(f)_{Q_{\rho 2^{-i}}(x_0)}$ for non-Lebesgue point $x_0$ of $f$, which is measure zero. Thus our redefinition of $f$  a  version of $f$ by ``killing" its non-Lebesgue point. Thus we have for $0<r\leq \diam(\Omega_T)$,
$$ \abs{ f(x_0) -(f)_{Q_r(x_0)} } \lesssim r^{-\frac{\lambda}{p}} \normp{f}{\Campanato{p,\lambda}{\Omega_T}} .$$
By taking $r=\diam(\Omega_T)$ and using H\"older inequality, above estimate also gives us the following $L^\infty$-bound:
$$ \sup_{\Omega_T} |f| \lesssim \diam(\Omega_T)^{-\frac{\lambda}{p}} \normp{f}{\Campanato{p,\lambda}{\Omega_T}} + |(f)_{\Omega_T}|  \lesssim   \normp{f}{\Campanato{p,\lambda}{\Omega_T}} + \normp{f}{L^p(\Omega_T)} .$$
Consider $x,y\in \Omega_T$ with $\delta(x,y)=R<\frac12\diam(\Omega_T)$, 
and denote $E = Q_{2R}(x)\cap Q_{2R}(y)\cap \Omega_T \supseteq Q_R(x)\cap\Omega_T$, we can compute that 
\begin{align*}
\abs{ (f)_{Q_{2R}(x)} - (f)_{Q_{2R}(y)} }^p & \lesssim \dashint_E \abs{(f)_{Q_{2R}(x)} - f(z) }^p dz + \dashint_E \abs{(f)_{Q_{2R}(y)} - f(z) }^p dz \\
& \lesssim R^{-(n+2a)} R^{n+2a-\lambda} \normp{f}{\Campanato{p,\lambda}{\Omega_T}}^p = R^{-\lambda} \normp{f}{\Campanato{p,\lambda}{\Omega_T}}^p 
\end{align*}
and hence
\begin{align*}
\abs{ f(x) - f(y) } & \leq \abs{f(x) - (f)_{Q_{2R}(x)}} + \abs{ f(y) - (f)_{Q_{2R}}(y)}  + \abs{ (f)_{Q_{2R}(x)} - (f)_{Q_{2R}(y)} } \\
& \lesssim \normp{f}{\Campanato{p,\lambda}{\Omega_T}} R^{-\frac{\lambda}{p}} + \abs{ (f)_{Q_{2R}(x)} - (f)_{Q_{2R}(y)} } \lesssim \normp{f}{\Campanato{p,\lambda}{\Omega_T}} R^{-\frac{\lambda}{p}} .
\end{align*}
If instead $\delta(x,y)\geq \frac12\diam(\Omega_T)$, we can estimate discrete displacement of $f$ by
$$ \abs{f(x) - f(y)} \leq 2\normp{f}{L^\infty(\Omega_T)} \lesssim \normp{f}{\Campanato{p,\lambda}{\Omega_T}} \rbracket{\frac{\delta(x,y)}{\diam(\Omega)}}^{-\frac{\lambda}{p}} .$$
\end{proof}

\begin{lemma}[$\widetilde{M}_0$ on $\Morrey{p,\lambda}{\Omega_T}$]\label{Lem: M0Morrey} Maximal operator $\widetilde{M}_0$ is continuous on $\Morrey{p,\lambda}{\Omega_T}$ for $p\in(1,\infty)$ and $\lambda \in(0,n+2a]$.
\end{lemma}
\begin{proof} 
For any $f\in\Morrey{p,\lambda}{\Omega_T}$, we decompose $f=f_1 + f_2$ where $f_1 = f \cdot \II \{Q_{2r}(z_0)\}$. Then
$$ \int_{z\in Q_r(z_0)} (\widetilde{M}_0 f)^p dz \leq 2^p  \int_{z\in Q_r(z_0)} (\widetilde{M}_0 f_1)^p dz +  \int_{z\in Q_r(z_0)} (\widetilde{M}_0 f_2)^p dz = 2^p(I_1 + I_2).$$
We use the fact that  maximal operator $\widetilde{M}_0$ is continuous on $L^p$, see \cite{Stein70}. As a consequence,
$$ I_1 \leq C \int_{z\in \bbR^{n+1}}  |f_1(z)|^p dz \leq C \int_{Q_{2r}(z_0)}  |f(z)|^p dz \leq C r^{n+2a-\lambda}\normp{f}{\Morrey{p,\lambda}{\Omega_T}}^p .$$
On the other hand, for any $\rho>0$, we can compute that for $z\in Q_r(z_0)$,
\begin{align*} & \frac{1}{\rho^{n+2a}}\int_{Q_\rho(z)} |f_2(z')| dz' = \frac{1}{\rho^{n+2a}}\int_{\delta(z,z')<\rho,\delta(z_0,z')>2r} |f(z')| dz'   \\
\leq & c \rho^{-\frac{\lambda}{p}} \rbracket{\rho^{\lambda - (n+2a)} \int_{\delta(z,z')<r} |f(z')|^p dz' }^{\frac1p} \leq c r^{-\frac{\lambda}{p}} \normp{f}{\Morrey{p,\lambda}{\Omega_T}},
\end{align*}
where we have used the fact that $\delta(z,z')<\rho,\delta(z_0,z')>2r,\delta(z_0,z)<r$ implies that above integral is nontrivial only when $\rho\geq r$. As a consequence, above inequality allows us to estimate $I_2\leq C r^{n+2a-\lambda} \normp{f}{\Morrey{p,\lambda}{\Omega_T}}^p$. Hence, 
$$ \normp{\widetilde{M}_0f}{\Morrey{p,\lambda}{\Omega_T}}^p = \sup_{r>0,z\in\Omega_T} r^{\lambda-(n+2a)}\normp{M_0f}{L^p(\Omega_T\cap Q_r(z))}^p \leq C \normp{f}{\Morrey{p,\lambda}{\Omega_T}}^p  ,$$ and the result follows.
\end{proof}

\begin{lemma}[Riesz potential estimates between fractional parabolic Morrey space, Regularity I]\label{Lem: Riesz potential1} 
Let $\pRieszP{\beta}{f}$ denote the fractional parabolic Riesz potential of order $\beta\in[0,n+2a]$,
$$ \pRieszP{\beta}{f}=\int_{\bbR^{n+1}} \frac{f(y,s)}{\delta((x,t),(y,s))^{n+2a-\beta}}dyds ,$$
where $f\in L^p(\bbR^{n+1})$ and fractional parabolic distance $\delta((x,t),(y,s))=\max(|x-y|,|t-s|^{\frac1{2a}})$. Then we have the following improved regularity:
\begin{itemize}
\item For any $\beta>0, 0<\lambda\leq n+2a, 1<p<\frac{\lambda}{\beta}$, if $f\in L^p(\bbR^{n+1})\cap \Morrey{p,\lambda}{\bbR^{n+1}}$, then $\pRieszP{\beta}{f}\in L^{\tilde{p}}(\bbR^{n+1})\cap \Morrey{\tilde{p},\lambda}{\bbR^{n+1}}$ with $\tilde{p} = \frac{p\lambda}{\lambda - p\beta}$. 
\item For any $0<\beta<\lambda\leq n+2a$, if $f\in L^1(\bbR^{n+1})\cap\Morrey{1,\lambda}{\bbR^{n+1}}$, then $\pRieszP{\beta}{f}\in L^{\frac{\lambda}{\lambda-\beta},*}(\bbR^{n+1})\cap \wMorrey{\frac{\lambda}{\lambda-\beta},\lambda}{\bbR^{n+1}}$.
\end{itemize}
\end{lemma} 
\emph{Remarks:} Let $I_\beta(f)$ denote the parabolic Riesz potential of order $\beta\in[0,n+2]$, that is,
$$ I_\beta(f) = \int_{\bbR^{n+1}} \frac{f(y,s)}{\max(|x-y|,|t-s|^{\frac12})^{n+2-\beta}} dyds  $$
then we have an analogous improved regulairty: 
\begin{itemize}
\item For any $\beta>0, 0<\lambda\leq n+2, 1<p<\frac{\lambda}{\beta}$, if $f\in L^p(\bbR^{n+1})\cap L^{p,\lambda}(\bbR^{n+1})$, then $I_{\beta}(f)\in L^{\tilde{p}}(\bbR^{n+1})\cap L^{\tilde{p},\lambda}(\bbR^{n+1})$ with $\tilde{p} = \frac{p\lambda}{\lambda - p\beta}$. 
\item For any $0<\beta<\lambda\leq n+2$, if $f\in L^1(\bbR^{n+1})\cap L^{1,\lambda}(\bbR^{n+1})$, then $I_{\beta}(f) \in L^{\frac{\lambda}{\lambda-\beta}}_*(\bbR^{n+1})\cap L^{\frac{\lambda}{\lambda-\beta},\lambda}_*(\bbR^{n+1})$.
\end{itemize}

\begin{proof} It can be considered the analogue
of the Sobolev inequality in fractional parabolic Morrey Spaces. The proof follows the same argument of improved regularity of Morrey space for classical Riesz potential, see \cite{Adams15}. Recall that the maximal operator $M_\alpha f, 0<\alpha<n+2a$, is defined as follows:
$$ M_\alpha f(z) = \sup_{r>0} r^\alpha \dashint_{Q_r(z)} |f(z')| dz', \ \ z=(t,x),z'=(t',x')\in\bbR^{n+1} .$$
We use the Hedberg trick \cite{Hedberg72} and claim that 
\begin{equation}\label{Eqn: Hedberg} \pRieszP{\beta}{f}(z) \leq c [M_{\frac{\lambda}{p}}f(z)]^{\frac{\beta p}{\lambda}} [M_0 f(z)]^{1-\frac{\beta p}{\lambda}} .\end{equation}
To show this, we decompse $f=f_1 + f_2$ where $f_1 = f \cdot \II \{Q_{2r}(z)\}$ and then $\pRieszP{\beta}{f}=\pRieszP{\beta}{f_1}+\pRieszP{\beta}{f_2}=J_1+J_2$ where
\begin{align*}
J_1 & = \sum_{k=0}^\infty \int_{\delta(z,z')\in(2^{-k}r,2^{-k+1}r)} \delta(z,z')^{\beta-(n+2a)} f(z') dz' \\
& \leq \sum_{k=0}^\infty (2^{-k}r)^{\beta - (n+2a)} (2^{-k+1}r)^{n+2a} M_0f(x) \leq C r^\beta M_0f(x), \\
J_2 & = \sum_{k=1}^\infty \int_{\delta(z,z')\in(2^kr,2^{k+1}r)} \delta(z,z')^{\beta-(n+2a)} f(z') dz' \\
& \leq \sum_{k=1}^\infty (2^k r)^{\beta-(n+2a)} (2^{k+1}r)^{n+2a-\frac{\lambda}{p}} M_{\frac{\lambda}{p}} f(x) \leq C r^{\beta - \frac{\lambda}{p}} M_{\frac{\lambda}{p}}f(x)
\end{align*}
for $0<\beta<\frac{\lambda}{p}$. Now the claim is verified once we choose $r=[M_{\frac{\lambda}{p}}f(x) / M_0f(x)]^{\frac{p}{\lambda}}$. Moreover, observe that $M_{\frac{\lambda}{p}}f \leq (M_\lambda |f|^p)^{\frac1p} \leq \normp{f}{\Morrey{p,\lambda}{\Omega_T}}$. Thus we raise equation \eqref{Eqn: Hedberg} to power $\tilde{p}$ and integrate over fractional parabolic cylinder $Q_r(x_0)$ to obtain
$$ \int_{Q_r(x_0)} |\pRieszP{\beta}{f}(z) |^{\tilde{p}} dz \leq C   \normp{f}{\Morrey{p,\lambda}{\Omega_T}}^{\frac{\beta p}{\lambda} \tilde{p}} \int_{Q_r(x_0)} |M_0 f(z)|^p dz ,$$
which implies that $\pRieszP{\beta}{f}\in L^{\tilde{p}}$. Moreover, above inequality with Lemma \ref{Lem: M0Morrey} also implies that 
$$ r^{\lambda-(n+2a)} \int_{Q_r(x_0)} |\pRieszP{\beta}{f}(z) |^{\tilde{p}} dz  \leq C \normp{f}{\Morrey{p,\lambda}{\Omega_T}}^{\frac{\beta p}{\lambda} \tilde{p}} \normp{f}{\Morrey{p,\lambda}{\Omega_T}}^p $$
and thence $\pRieszP{\beta}{f}\in \Morrey{p,\lambda}{\Omega_T}$.
\end{proof}

Here we also present the proof of increased regularity by Riesz potential, which acts as a prototype for the proof of H\"older regularity in Theorem \ref{Prop: Reg2-intro}.
\begin{lemma}[Riesz potential estimates between fractional parabolic Morrey space, Regularity II]\label{Lem: Riesz potential2} We claim that if $f(x,t)\in \Morrey{p,\lambda}{\bbR^{n+1}}$ with $1<p<\infty, 0<\lambda\leq n+2a, \frac{\lambda}{p}<\beta<1$, and $\beta p\geq n+2a$, then $\pRieszP{\beta}{f}(x,t) \in C^{\beta - \frac{\lambda}{p}}_{loc, \delta}$, i.e. $|\pRieszP{\beta}{f}(x,t)-\pRieszP{\beta}{f}(x',t')|\leq C(n,s) \delta((x,t),(x',t'))^{\beta - \frac{\lambda}{p}}$ whenever $\delta((x,t),(x',t'))< 1$.
\end{lemma}
\begin{proof}
We follow the same argument for estimate of classical Riesz potential in Lemma 3.5 of \cite{Miyakawa90}. By Lemma \ref{Lem: Campanato}, it is equivalent to show that $\pRieszP{\beta}{f}(z)\in \Campanato{1,\frac{\lambda}{p}-\beta}{Q_1(z_0)} = \Campanato{p,\lambda-p\beta}{Q_1(z_0)}=C^{\beta-\frac{\lambda}{p}}_\delta(Q_1(z_0))$. 
It suffices to show that
$$ \dashint_{z\in Q_r(z_0)} \dashint_{z'\in Q_r(z)} \abs{ \int_{\xi \in\bbR^{n+1}} \frac{f(\xi)}{\delta(z',\xi)^{n+2a-\beta}} - \frac{f(\xi)}{\delta(z,\xi)^{n+2a-\beta}} } \lesssim  r^{\beta   - \frac{\lambda}{p}}. $$
We claim the following technical result that if $0<\beta<1$, then
\begin{equation}\label{Eqn: Tech1} \int_{\bbR^{n+1}} \abs{\frac{1}{\delta(z,w)^{n+2a-\beta}} - \frac{1}{\delta(0,z)^{n+2a-\beta}}} dz \lesssim \delta(0,w)^\beta ,\end{equation}
which allows us to finish the proof of this lemma by the following argument (with $\eta = z'-z\in Q_r(0), y=z'-\xi$),
\begin{align*}
&\int_{z'\in Q_r(z)} \int_{\xi \in\bbR^{n+1}} |f(\xi)| \abs{\frac{1}{\delta(z',\xi)^{n+2a-\beta}} - \frac{1}{\delta(z,\xi)^{n+2a-\beta}} } \\
= & \int_{y\in\bbR^{n+1}} \abs{ \frac{1}{\delta(0,y+\eta)^{n+2a-\beta}} - \frac{1}{\delta(y,0)^{n+2a-\beta}}} \int_{\xi\in\bbR^{n+1}}|f(\xi)|\II\cbracket{\delta(\xi,z-y)<r} \\
  \lesssim & \normp{f}{L^p(Q_r(z-y))} r^{(n+2a)(1-\frac1p)} r^\beta \\
  \leq & \normp{f}{\Campanato{p,\lambda}{\bbR^{n+1}}} r^{n+2a-\frac{\lambda}{p}+\beta}.
\end{align*}
To prove inequality \ref{Eqn: Tech1}, we decompose the left hand side integral into inner and outer components. It is easy to verify that metric $\delta$ on $\bbR^{n+1}$ satisfies Triangle Inequality, and hence the inner component
\begin{align*} J_{inner} & = \int_{z\in Q_{2\delta(0,\omega)}(0)} \abs{\frac{1}{\delta(z,w)^{n+2a-\beta}} - \frac{1}{\delta(0,z)^{n+2a-\beta}}}   dz \\
& \leq \int_{z\in Q_{2\delta(0,\omega)}(0)} \frac{1}{\delta(z,w)^{n+2a-\beta}} + \frac{1}{\delta(0,z)^{n+2a-\beta}}   dz \\
&  \leq \int_{z\in Q_{3\delta(0,\omega)}(0)} \frac{1}{\delta(z,0)^{n+2a-\beta}} \lesssim \delta(0,w)^\beta.
\end{align*}
To estimate the outer component, we notice that by triangle inequality $\delta(0,z)\geq 2\delta(0,w)$ implies that $\delta(0,z+\theta w)\geq\delta(0,z)-\delta(0,w)\geq \frac12\delta(0,z) $. As a consequence, when  $\delta(0,z)\geq 2\delta(0,w)$, we define $\phi(\theta)(w_x,w_t)=(\theta w_x, \theta^{2a} w_t)$, which is ``linear" interpolation operator  for $\theta\in[0,1]$ in fractional parabolic space.
\begin{align*}  \abs{\frac{1}{\delta(z,w)^{n+2a-\beta}} - \frac{1}{\delta(0,z)^{n+2a-\beta}}}  \leq&~ \int_0^1 \abs{ \frac{d \delta^{\beta-(n+2a)}}{d\theta}  (z+\phi(\theta)w) } d\theta\\
 \lesssim &~\frac{\delta(0,w)}{\delta(0,z)^{n+2a+1-\beta}} .
\end{align*}
Hence, we integrate above inequality over  $\delta(0,z)\geq 2\delta(0,w)$ and obtain estimate for outer component
\begin{align*}
J_{outer} &= \int_{ \delta(0,z)\geq 2\delta(0,w)} \abs{\frac{1}{\delta(z,w)^{n+2a-\beta}} - \frac{1}{\delta(0,z)^{n+2a-\beta}}} dz \\
&\lesssim \delta(0,w) \int_{ \delta(0,z)\geq 2\delta(0,w)} \frac{1}{\delta(0,z)^{n+2a+1-\beta}} dz \\
& \lesssim \delta(0,w)^\beta,
\end{align*}
where we have used the fact that $\beta\in(0,1)$ to make sure that the last integral converges.
\end{proof}



\section*{Acknowledgments}

F. Lin is partially supported by NSF DMS $2247773$. Y. Sire is partially supported by DMS NSF grant $2154219$, `` Regularity {\sl vs} singularity formation in elliptic and parabolic equations'' . Y. Zhou is supported in part by the Fundamental Research Funds for the Central Universities.

\bibliographystyle{alpha}
\bibliography{RefDatabase,RefDatabase1,biblio,Refdatabase2} 
\end{document}